\documentclass[10pt]{amsart}
\usepackage{amssymb, latexsym}
\usepackage{graphicx}
\usepackage[scriptsize, up]{caption}
\usepackage{pdfsync, wrapfig}
\usepackage{xcolor}
\definecolor{dnrbl}{rgb}{0,0,0.5}
\definecolor{dnrgr}{rgb}{0,0.5,0}
\definecolor{dnrre}{rgb}{0.5,0,0}
\usepackage[colorlinks=false, citecolor=dnrgr, linkcolor=dnrre, urlcolor=dnrbl] {hyperref}
\usepackage{pgf,tikz}

\theoremstyle{plain}
\newtheorem{thm}{Theorem}[section]

\newtheorem{lem}[thm]{Lemma}
\newtheorem{coro}[thm]{Corollary}

\numberwithin{equation}{section}

\newcommand{\Nat}{\mathbb{N}}

\newcommand{\restr}{\upharpoonright}  

\newcommand{\un}{\uparrow} 
\newcommand{\de}{\downarrow} 

\newcommand{\ml}{Martin-L\"{o}f }

\begin{document}
\title[Universality, enumerability and prefix-free complexity]{Universal 
computably enumerable sets and initial segment prefix-free complexity}

\author{George Barmpalias}
\address{{\bf George Barmpalias:} 
State Key Lab of Computer Science, 
Institute of Software,
Chinese Academy of Sciences,
Beijing 100190, 
China.}
\email{barmpalias@gmail.com}
\urladdr{\href{http://www.barmpalias.net}{http://www.barmpalias.net}}
\date{{\em First version:} 1 October 2011. {\em This version:} \today}

\thanks{Barmpalias
 was supported by a
research fund for international young scientists
No.\ 611501-10168 and
an {\em International Young Scientist Fellowship} 
number 2010-Y2GB03 from the Chinese Academy of 
Sciences. Partial support was also obtained by
the {\em Grand project: Network Algorithms and Digital Information} of the
Institute of Software, Chinese Academy of Sciences. Special thanks go to the referees
who carefully read a previous draft and suggested many improvements.}

\begin{abstract} 
We show that there are Turing complete computably
enumerable sets of arbitrarily low non-trivial initial segment prefix-free complexity.
In particular, given any computably enumerable set $A$ with non-trivial 
prefix-free initial segment
complexity, there exists a Turing complete computably enumerable set $B$
with complexity strictly less than the complexity of $A$. 
On the other hand it is known
that sets with trivial initial segment prefix-free complexity are not Turing complete.

Moreover we give a generalization of this result for any finite collection 
of computably enumerable sets $A_i, i<k$ with non-trivial initial
segment prefix-free complexity.
An application of this gives a negative answer to a question
from \cite[Section 11.12]{rodenisbook} and \cite{MRmerstcdhdtd} 
which asked for minimal pairs in
the structure of the c.e.\ reals ordered by their initial segment
prefix-free complexity.

Further consequences concern various notions of degrees of randomness.
For example, the Solovay degrees and the $K$-degrees
of computably enumerable reals and computably enumerable sets
are not elementarily equivalent.
Also, the degrees of randomness of c.e.\ reals based on plain and prefix-free complexity
are not elementarily equivalent; the same holds for the degrees of c.e.\ sets.
\end{abstract}
\maketitle
\section{Introduction}
The interplay between the information that can be coded into
an infinite binary sequence and its initial segment complexity
has been the subject of a lot of research in the last ten years.
A rather influential result 
from \cite{MRtrivrealsH}
that spawned a renewed interest in this area
was that sequences with very easily describable initial segments
cannot compute the halting problem.
Moreover the method that was used to establish it,
often referred to as the decanter method,
was novel and inspired much of the deeper work in this area.
We show that although a universal computably enumerable set
does not have trivial initial segment complexity, it can have
arbitrarily low non-trivial initial segment complexity.
Moreover our method is dual to the decanter method 
and in this sense 
the present paper can be seen as a missing companion to
\cite{MRtrivrealsH}.

We start with a brief overview of Kolmogorov complexity
in Section \ref{subse:kolcomran}
and measures of relative randomness in Section \ref{subse:mearelra}
with a special attention
to the topics around our results.
In Section \ref{subse:trivial} we discuss the class of sequences
with trivial initial segment complexity along with the motivation of
our results, which are presented in Section \ref{subse:resul}.
A number of applications are given in
Section \ref{subse:applic} and Section \ref{subse:otherwred}
discusses connections of the present work with research on other 
reducibilities that are related to Kolmogorov complexity.
In Section \ref{se:prothlowisa} we introduce the main technical tools that
are required for the proofs of our results and Sections
\ref{subse:profthA} and \ref{subse:profthB} contain the proofs of the two main results
respectively.

\subsection{Kolmogorov complexity and randomness}\label{subse:kolcomran}
A standard measure of the complexity of a finite string 
was introduced by
Kolmogorov in \cite{MR0184801}.
The basic idea behind this approach 
is that simple strings have short descriptions relative
to their length
while complex or random strings are hard to describe
concisely.
Kolmogorov formalized this idea using the theory of computation.
In this context, Turing machines play the role of our idealized computing devices, and
we assume that there are Turing machines capable of simulating any mechanical
process which proceeds in a precisely defined and algorithmic manner.
Programs can be identified with binary strings. A string $\tau$ is said to be a
description of a string $\sigma$ with respect to a Turing machine
$M$ if this machine halts when given program $\tau$ and
outputs $\sigma$. Then the Kolmogorov complexity of $\sigma$ with respect to $M$ 
(denoted by $C_M(\sigma)$) is the length
of its shortest description with respect to $M$. 
It can be shown that there exists an \emph{optimal} prefix-free machine $V$, i.e.\
a machine which gives optimal complexity for all strings, up to a certain constant number of bits.
This means that for each Turing machine $M$ there exists a constant $c$
such that $C_V(\sigma)< C_M(\sigma)+c$ for all finite strings $\sigma$.

When we come to consider randomness for infinite strings, it becomes important to
consider machines whose domain satisfies a certain condition;
the machine $M$ is called \emph{prefix-free} if
it has prefix-free domain (which means that no program 
for which the machine halts and gives output is an initial segment of another).
The complexity of a string $\sigma$ with respect to a prefix-free machine $M$ is
denoted by $K_M(\sigma)$.
As with the case of plain Turing machines, there exists an \emph{optimal} 
prefix-free machine $U$.
This means that for each prefix-free machine $M$ there exists a constant $c$
such that $K_U(\sigma)< K_M(\sigma)+c$ for all finite strings $\sigma$.

According to the above discussion, both in the case of plain 
or prefix-free Turing machines the choice of the underlying optimal machine does
not change the complexity distribution significantly.
Hence the theories of plain and prefix-free complexity can be developed without loss of generality,
based on fixed underlying optimal plain and prefix-free machines $V, U$. We let $C=C_V$ and
$K=K_U$.

In order to define randomness for infinite sequences, we consider the complexity of all
finite initial segments. A finite string $\sigma$ is said to be  $c$-{\em incompressible} 
if $K(\sigma)\geq |\sigma|-c$.
Levin \cite{MR0366096}
and Chaitin \cite{MR0411829}  defined an infinite binary sequence  $X$ to be random if
there exists some constant $c$ such that all of its initial segments are $c$-incompressible.
By identifying subsets of $\Nat$ with their characteristic sequence we can also talk about
randomness of sets of numbers.
This definition of randomness of infinite sequences is independent of the choice of
underlying optimal prefix-free machine, and coincides with other definitions of randomness like
 the definition given by \ml in \cite{MR0223179}. 
The coincidence of the randomness notions resulting from various different 
approaches may be seen as evidence of a robust and natural theory.

\subsection{Measures of relative randomness}\label{subse:mearelra}
Once a solid definition of initial segment complexity and randomness is in place, 
it is often desirable to have a way to compare two infinite binary sequences in this respect.
One of the early measures of relative initial segment complexity was developed by
Solovay in \cite{Solovay:75} especially for the computably enumerable (c.e.)
reals. These are binary expansions of the real numbers in the unit interval which are limits
of increasing computable sequences of rationals. The Solovay reducibility
gave a formal way to compare c.e.\ reals with respect to
the difficulty of getting good approximations to them.
Solovay showed in \cite{Solovay:75}
that the induced degree structure has a complete element which
contains exactly the random c.e.\ reals. The Solovay degrees of c.e.\ 
reals where further studied in \cite{DHN,DowneyHL07} (see \cite[Section 9.5]{rodenisbook} for an overview).

Downey, Hirschfeld and LaForte \cite{MR2030512} introduced and studied a number of
other measures of relative initial segment complexity that are not restricted  to the c.e.\ reals.
Most of them are extensions of the Solovay measure of relative complexity.
For example, they defined $A\leq_K B$ if $\exists c\forall n\ (K(A\restr_n)\leq K(B\restr_n)+c)$;
in other words, if the prefix-free complexity of each initial segment of $A$ is bounded by the prefix-free 
complexity of the corresponding initial segment of $B$, modulo a constant.
This reducibility, already implicit in \cite{Solovay:75}, 
is a proper extension of the Solovay reducibility on the c.e.\ reals
and was further studied  in \cite{DingyuDow, milleryutran, milleryutran2}
with a special attention to random sequences and in \cite{MR2199198, MRmerstcdhdtd, BV2010},
\cite[Section 5]{Barstris}
with more focus on local properties. A lot of these results refer to
the degree structure that is induced by $\leq_K$, the $K$-degrees.
A version of $\leq_K$ for plain Kolmogorov complexity
was also defined in \cite{MR2030512}, which induces the structure of the $C$-degrees.
In particular, $A\leq_C B$ if $\exists d\forall n\ (C(A\restr_n)\leq C(B\restr_n)+d)$.

\subsection{Trivial initial segment complexity and Turing degrees}\label{subse:trivial}
A string $\sigma$ that has  prefix-free complexity as low as
the prefix-free complexity of the sequence of 0s of the same length may be regarded as trivial.
Indeed, if we consider the prefix-free complexity of a string as a measure of the information 
that is coded in the string, in this case there is no information coded in the bits of the sequence.
The infinite sequences whose initial segments have trivial prefix-free complexity
are known as the $K$-trivial sequences. Formally, $X$ is $K$-trivial
if $\exists c\forall n\ (K(A\restr_n)\leq K(n)+c)$, where we may identify $K(n)$ with
$K(0^n)$.
Surprisingly, there are noncomputable $K$-trivial sequences and this was 
already proved in \cite{Solovay:75}.
Note that the $K$-trivial sequences are the contents of the least element
in the $K$-degrees that were discussed in Section \ref{subse:mearelra}.

An interesting question that motivated a lot of later research
was the following.
\begin{equation}\label{eq:quesmkthal}
\parbox{9cm}{How much information can be encoded in an infinite binary sequence
with very simple initial segments?}
\end{equation}
In particular, is it possible to encode a Turing complete problem into a
$K$-trivial sequence.
A particularly simple construction of a noncomputable $K$-trivial c.e.\ set
that was presented in \cite{MRtrivrealsH} made this possibility
plausible. However in the same paper it was shown that this is not the case.
In particular, if an oracle $A$ computes the halting problem then for each
constant $c$ there are initial segments $\sigma$ of $A$ such that $K(\sigma)>K(|\sigma|)+c$.
The proof of this result was quite novel, and along with its extensions
it became known as {\em the decanter method}.
Hirschfeldt and Nies extended this method in \cite{MR2166184}
and showed that the amount of information that can be coded into $K$-trivial
sequences is in fact quite limited. Quite interestingly, they also showed that
$K$-triviality is downward closed with respect to Turing reductions.
We refer to \cite[Section 11.4]{rodenisbook} and \cite[Section 5]{Ottobook}
for detailed presentations of the decanter method.

\subsection{Motivation and results}\label{subse:resul}
In this paper we revisit question (\ref{eq:quesmkthal})
by examining the possibility of coding considerable information
in an infinite sequence with initial segments of very low but not
necessarily trivial prefix-free complexity.
We initially focus in the special case of c.e.\ sets,
where Turing completeness provides a notion of
maximality of information that can be coded.
Hence we may ask the following question.
\begin{equation}\label{eq:quescomplrel}
\parbox{9cm}{How low can the initial segment prefix-free complexity of a Turing complete
computably enumerable set be?}
\end{equation}

How can we qualify  the notion of 
`low initial segment complexity' in question 
(\ref{eq:quescomplrel})?
Note that modulo an additive constant, 
$K(n)$ is a lower bound on the complexity of the first $n$
bits of any infinite sequence.
Since the $K$-trivial sequences are ruled out by the result in
 \cite{MRtrivrealsH}, we turn our attention to sequences whose initial
 segment prefix-free complexity may deviate from the lower bound $K(n)$
 but is still quite low. One way we could try to make this lowness condition precise
 is to look among sequences $A$ such that $K(A\restr_n)-K(n)$ is bounded 
 from above by a very
 slow growing function $g$, as it is shown in (\ref{eq:gapfun}).
 \begin{equation}\label{eq:gapfun}
 \exists c\forall n\ (K(A\restr_n)\leq K(n)+g(n)+c)
 \end{equation}
 The notion of `slow growing' may be quantified through the arithmetical
 hierarchy of complexity. For example there are $\Delta^0_3$ 
 unbounded nondecreasing functions that are dominated by all $\Delta^0_2$
 functions with the same properties. In this sense, as the  rate
 of growth of a function is reduced (but remains nontrivial)
 the arithmetical complexity of it increases.
Let us first consider nondecreasing functions $g$.
 In \cite{BienvenuMN11, Baarsbarmp} it was shown that if $g$ is 
 nondecreasing, unbounded and $\Delta^0_2$ then there is a large 
 uncountable collection of oracles $A$ that satisfy (\ref{eq:gapfun}).
 Hence a class that includes functions with these properties
 is not sufficiently restrictive for our purpose and we need to look
 in higher complexity classes.
 On the other hand in \cite{MR2199198, Baarsbarmp}
 it was shown that there are nondecreasing unbounded
 functions $g$ in $\Delta^0_3$
 such that any set $A$ that satisfies
 (\ref{eq:gapfun}) is $K$-trivial. 
 Moreover allowing functions that may decrease occasionally
 introduces similar problems. For example, 
 it was shown in \cite[Section 5]{BV2010}
 that there is a $\Delta^0_2$ function $g$ such that $\lim_n g(n)=\infty$ and
 any c.e.\ set which satisfies (\ref{eq:gapfun}) is $K$-trivial.
 Hence condition (\ref{eq:gapfun}) in combination with standard
 ways to quantify the rate of growth of the function $g$ is not a fruitful way
 to formalize the notion of `low nontrivial initial segment complexity'.
 
 Another approach is to compare the initial segment complexity of
 a c.e.\ set with the complexity of other sets.
 Although this would not give us an absolute notion of low nontrivial complexity,
 an answer of the type `lower than the complexity of any sequence with nontrivial complexity'
 to the question (\ref{eq:quescomplrel})
 would be definitive. The existence of minimal $K$-degrees is an open problem,
 but since this question refers to c.e.\ sets, such a positive answer is
 still not possible.
 Indeed, it was shown in \cite{BV2010} that there is a $\Delta^0_2$ set $B$ which is not $K$-trivial
 but every c.e.\ set with $A\leq_K B$ is $K$-trivial. In other words the 
 initial segment complexity of $B$ does not bound the complexity of any c.e.\ set with nontrivial
 initial segment complexity.
 This shows that the comparison needs to involve the complexities of c.e.\ sets and not 
 arbitrary sequences. In this sense, the best possible answer to question
 (\ref{eq:quescomplrel}) would be the existence of Turing complete c.e.\ sets
 with initial segment complexity strictly lower than the complexity of any given c.e.\ set
 that is not $K$-trivial. Our first result establishes exactly this.
 \begin{thm}\label{th:lowiscom}
Let $A$ be a computably enumerable set which is not $K$-trivial.
Then there exists a computably enumerable set $B$
such that $B\equiv_T\emptyset'$ and $B<_K A$.
\end{thm}
The proof of Theorem \ref{th:lowiscom} involves a very sparse coding
of complete information, which produces a sequence with very simple
initial segments, in the sense of the prefix-free complexity.
A crucial part of the argument is the exploitation of 
the fact that the given set is c.e.\ and has nontrivial initial segment
prefix-free complexity. In this sense Theorem \ref{th:lowiscom} is dual to the main result of
\cite{MRtrivrealsH} that $K$-trivial sets are incomplete. More
generally, the decanter method that was developed in 
\cite{MRtrivrealsH} is a tool for exploiting the lack of complexity of a set
in order to deduce additional properties. 
 The method used in the proof of
 Theorem \ref{th:lowiscom} is a tool
 for exploiting the complexity of a sequence 
 (in combination with an effective approximation to it)
 in order to absorb the complexity of a coding
 procedure. In this sense the two methods are dual.
 
 It is instructive to compare Theorem \ref{th:lowiscom}
 with condition (\ref{eq:gapfun}). If we wish to express
 our result in these terms we can set $g(n)=K(A\restr_n)-K(n)$.
 We note that $g(n)$ will be occasionally decreasing.
 In fact, it is well known that for every c.e.\ set $A$ the
 $\liminf(K(A\restr_n)-K(n))$ is finite. In other words, c.e.\ sets
 are infinitely often $K$-trivial (see \cite[Section 2]{BV2010} for a proof
 and a general discussion about infinitely often $K$-trivial sets). 
 This observation gives some idea about the challenges
 of implementing the coding that is required in Theorem \ref{th:lowiscom}
 as well as the qualification of the idea of `low initial segment complexity'
 for c.e.\ sets.
 
 Our second result is a generalization of Theorem \ref{th:lowiscom}
 to any finite collection of c.e.\ sets with nontrivial initial segment
 prefix-free complexity. We state it and prove it for the special 
 case of two c.e.\ sets since
 the more general version may be obtained trivially
 and effectively by an iterated application.
 \begin{thm}\label{th:lowismulticom}
Let $A, D$ be computably enumerable sets which are not $K$-trivial.
Then there exists a computably enumerable set $B$
such that $B<_K A$, $B<_K D$ and $B\equiv_T\emptyset'$.
\end{thm}
\noindent
 This extension has several applications that are discussed
 in Section \ref{subse:applic}, including the solution
 to an open question from \cite[Section 11.12]{rodenisbook}.
Moreover its proof goes considerably beyond
a routine adaptation of the special case established in
Theorem \ref{th:lowiscom}. As we elaborate in Section
\ref{subse:lackunifsol} the main obstacle is the lack of uniformity in the
complexities of the given c.e.\ sets.
 This can be better understood if we recall that $K$-trivial sets
 are infinitely often $K$-trivial.
 In particular, as we discuss in Section \ref{subse:applic},
 Theorem \ref{th:lowismulticom} shows that if two c.e.\ sets $A,D$ are 
 not $K$-trivial their initial segment complexity must rise simultaneously
 on some lengths. Hence despite the potential lack of uniformity in the oscilations
 of the complexity of two c.e.\ sets, there must be some uniformity on a local level i.e.\
 places where the complexities $K(A\restr_n), K(D\restr_n)$
 deviate from $K(n)$ simultaneously.

Finally, we would like to mention another approach that has been used in the
recent work by Ian Herbert with regard to reals of low initial segment complexity.
Let $K^A$ denote the prefix-free complexity with respect to oracle $A$.
Herbert studied the class of reals $A$ such that $K(n)\leq K^A(n)+f(n)+c$ for all $n$,
 where $c$ is a constant and $f$ is a slow growing function.
 This class is also a proper extension of the $K$-trivial reals.
 
\subsection{Applications}\label{subse:applic}
The first application concerns various local structures of the $K$-degrees.
The existence of minimal pairs of $K$-degrees was established in
\cite{MR2199198}, where two $\Delta^0_4$ sets forming a minimal pair
in this structure were constructed.
In \cite{MRmerstcdhdtd} a minimal pair of $\Sigma^0_2$ sets was 
presented  and in \cite[Section 3]{BV2010}
it was shown that there is a $\Sigma^0_2$ set that forms a minimal pair with all
$\Sigma^0_1$ sets in the $K$-degrees.
Theorem \ref{th:lowismulticom} implies that there are no
minimal pairs  in the structure of the $K$-degrees of c.e.\ sets. In
particular, there is no pair of $\Sigma^0_1$ sets that form a minimal pair of $K$-degrees.
This complements the existence results for minimal pairs in the $K$-degrees.

Downey and Hirschfeldt \cite[Section 11.12]{rodenisbook} 
as well as Merkle and Stephan \cite{MRmerstcdhdtd} asked if there is a
pair of c.e.\ reals that form a minimal pair in the $K$-degrees.
This question is particularly interesting since $\leq_K$ is often introduced
as a generalization of the Solovay reducibility, which is the standard measure of relative randomness
on the class of c.e.\ reals.
We show that Theorem \ref{th:lowismulticom} answers this question in the negative.
We need the following fact.
\begin{lem}\label{le:c.erealase}
If $A$ is a c.e.\ real such that $\emptyset<_K A$ then there exists
a c.e.\ set $B$ with $\emptyset <_K B\leq_K A$.
\end{lem}
\begin{proof}
Since $A$ is a c.e.\ real, it has a computable approximation $(A_s)$
according to which if $A(n)[s]=1$ and $A(n)[s+1]=0$
then there is some $i<n$ such that 
$A(i)[s]=0$ and $A(i)[s+1]=1$. A canonical encoding of the approximation
$(A_s)$ into a c.e.\ set $B$ can be achieved
based on the fact that for each $n$ the value of $A(n)[s]$
can only change at most $2^n$ times during the stages $s$.
The first bit of $B$ encodes the oscillations to $A(0)$,
the next $2$ bits encode $A(1)$, the next $2^2$ bits encode $A(2)$
and so on. In particular if $A(k)$ is encoded in the bits $(m, m+2^{k}]$
of $B$, upon each change in
$A(k)[s]$ during the stages $s$ we enumerate into $B$ the largest
element of $(m, m+2^{k}]$ that is not yet in $B$.
In this way we have $A\equiv_T B$ and $B\leq_T A$ through a
Turing reduction that uses at most $n$ bits of $A$ in the computation of
$n$ bits of $B$. Since $K$-triviality is a degree-theoretic property we have
$\emptyset <_K B$ and by the basic properties of $\leq_K$ on the c.e.\ reals
we also have $B\leq_K A$.
\end{proof}
\noindent
By Theorem \ref{th:lowismulticom} and
Lemma \ref{le:c.erealase} we get the desired result about minimal pairs.
\begin{coro}\label{coro:nomincere}
There are no minimal pairs in the $K$-degrees of c.e.\ reals.
\end{coro}

The separation of Solovay reducibility from $\leq_K$ on the c.e.\ reals
was already achieved in \cite{MR2030512}, where
a pair of c.e.\ reals $A,B$ was constructed such that $A\leq_K B$ 
but $A$ is not Solovay reducible to $B$. However 
these examples are artificial since they were obtained via diagonalization.
A more natural separation would be obtained by an elementary difference 
in the corresponding degree structures of c.e.\ reals. 
 This is provided by the existence of minimal pairs
which occurs in the Solovay degrees of c.e.\ reals by \cite{MR2030512}
but not in the $K$-degrees of c.e.\ reals by Corollary \ref{coro:nomincere}.
 The same holds for c.e.\ sets according to Theorem \ref{th:lowismulticom}.

\begin{coro}
The structures of the Solovay degrees and the $K$-degrees of
computably enumerable reals are not elementarily equivalent.
Moreover the same holds for the Solovay degrees and the $K$-degrees 
of computably enumerable sets.
\end{coro}
\noindent
Merkle and Stephan showed in \cite{MRmerstcdhdtd} that there
exist two c.e.\ sets that from a minimal pair with respect to $\leq_C$.
Hence Corollary \ref{coro:nomincere} also provides an elementary difference
between the $C$-degrees and the $K$-degrees of c.e.\ reals and c.e.\ sets.
\begin{coro}
The structures of the $C$-degrees and the $K$-degrees 
of c.e.\ reals are not elementarily equivalent.
Moreover the same holds for the corresponding structures
of c.e.\ sets.
\end{coro}

A final application of Theorem \ref{th:lowismulticom}
concerns the following question.
\begin{equation}\label{eq:comlohcom}
\parbox{10cm}{Is there a pair of sequences $X,Y$ which are not $K$-trivial
but $\min\{K(X\restr_n), K(Y\restr_n)\}-K(n)$ has a constant upper bound?}
\end{equation}
Theorem \ref{th:lowismulticom} in combination
with Lemma \ref{le:c.erealase} answers (\ref{eq:comlohcom}) in the negative
in the case where $X,Y$ are required to be computably enumerable reals.
\begin{coro}\label{coro:exactpair}
Suppose that $A_i$, $i<k$ is a finite collection of 
computably enumerable reals and
none of them is 
$K$-trivial. Then for all $c$ there exist $n$ such that 
$K(A_i\restr_n)>K(n)+c$ for all $i<k$.
\end{coro}
\noindent
We do not know the answer of (\ref{eq:comlohcom}) in general.

We conclude this section with a brief discussion on the topic of the
initial segment complexity of c.e.\ sets.
It would be interesting to locate elementary differences between the
$K$-degrees of c.e.\ reals and the $K$-degrees of c.e.\ sets.
This was done in \cite{DBLP:conf/cie/Barmpalias05} for the Solovay degrees
by showing that there are no maximal elements
in the Solovay degrees of c.e.\ sets. This line of research on the c.e.\ sets with respect
to reducibilities that are sensitive to initial segment complexity measures
was extended in \cite{Merklebarmp}. The quest
for elementary differences between
$K$-degrees of c.e.\ reals and the $K$-degrees of c.e.\ sets
lead to more general questions
regarding the c.e.\ sets in the
$K$-degrees and the $C$-degrees which were articulated in
a research proposal that was presented (along with several related results) 
in \cite{barak}.
An interesting product of this project is the following result from \cite{ChaitAnalogue}.
\begin{equation}\label{eq:maxcikc}
\parbox{10.5cm}{{\em There is a maximum in the $K$-degrees and the $C$-degrees of c.e.\ sets.}}
\end{equation}
In other words, there are c.e.\ sets with maximum initial segment complexity.
For the case of the plain complexity, a c.e.\ set $A$ has maximum
initial segment complexity if and only if the halting problem is reducible to it via
a Turing oracle computation where the oracle use is bounded by a linear function.
Moreover, it turned out that the above condition is equivalent to
$\forall n,\ C(A\restr_n)\geq \log n -c$ which is a well known property that was studied in
\cite{Barzlemma}. It follows from \eqref{eq:maxcikc} and 
\cite{DBLP:conf/cie/Barmpalias05} (see \cite{barak}  for more details) that the existence of a maximum
degree is an elementary difference between the $K$-degrees of c.e.\ sets and the $K$-degrees of c.e.\ reals.

\subsection{Related work on weak reducibilities}\label{subse:otherwred}
A method for exploiting the power of an oracle to achieve better compression of programs
(along with a computable approximation to it) has been used in the study of another
reducibility that is related to randomness and is called $\leq_{LK}$.
We say that $X\leq_{LK} Y$ if $\exists c\forall \sigma\ (K^Y(\sigma)\leq K^X(\sigma)+c)$.
In other words $X\leq_{LK} Y$ formalizes the notion that $Y$ can achieve an overall
compression of the strings that is at least as good as the compression achieved by $X$.
Moreover by \cite{millerdomi} it coincides with $X \leq_{LR} Y$ which denotes the relation
that every random sequence relative to $Y$ is also random relative to $X$.
The degree structure that is induced by $X\leq_{LK} Y$  has a least element that turns out to contain
exactly the $K$-trivial sequences.
In \cite{BarmpaliasM09} 
an argument was used that exploits  the compression power of nontrivial 
c.e.\ sets in the study of the structure of 
c.e.\ sets under $\leq_{LK}$.
A similar argument was used in \cite{Barmpalias:08}
in order to show that every $\Delta^0_2$ set 
with nontrivial compression power has uncountably many
predecessors with respect to $\leq_{LK}$.
In \cite{BarmpaliasCompress} this approach was further developed
in order to exhibit elementary differences between various local
structures of the $LK$ degrees and the Turing degrees.
We note that the arguments in these references work explicitly with 
$\leq_{LR}$ but can alternatively be implemented with
the equivalent $\leq_{LK}$.

However there are some differences between $\leq_K$ and $\leq_{LK}$,
the most important being that in $\leq_{LK}$ we usually work with oracle computations
while in $\leq_K$ we only work with descriptions.
It is quite remarkable that the triviality notion with respect to 
$\leq_K$ coincides with the triviality notion with respect to $\leq_{LK}$.
As soon as we consider sequences of non-zero $K$-degrees or $LK$-degrees,
the study of the two structures becomes less uniform.
A comparison of the arguments about the non-existence of
minimal pairs of $K$-degrees in this paper with
the corresponding arguments in \cite{BarmpaliasCompress} that refer to the $LK$ 
degrees shows that they follow a similar structure, yet various aspects need to be
addressed individually. We discuss the high level view of these arguments in Section \ref{se:concrem}.
 
\section{Preliminaries}\label{se:prothlowisa}
The main tool in the proof of these theorems is a method of coding information
into a set $B$ that is constructed, while keeping its initial segment 
complexity below
the complexity of a given c.e.\ set $A$ that is not $K$-trivial.
It is a method for exploiting the fact that a given set has a computable enumeration
and non-trivial initial segment complexity, for the purpose of coding. 
In particular, it allows to meet the conflicting requirements
$\emptyset'\leq_T B$ and $B\leq_K A$.

\subsection{Prefix-free machines}
For $B\leq_K A$ we need to build a prefix-free machine
that witnesses the relation of the two complexities.
Let $U$ be the optimal
prefix-free machine which underlies the 
prefix-free complexity $K$.
Hence $K=K_{U}$. This machine is optimal in the sense that given any other
prefix-free oracle machine $M$ there is a constant $c$ such that
$K(\sigma)\leq K_M(\sigma)+c$ for all strings $\sigma$.
The {\em weight} of a prefix-free set $S$ of strings, denoted $\mathtt{wgt}(S)$, is defined
to be the sum $\sum_{\sigma\in S}  2^{-|\sigma|}$. The
{\em weight} of a prefix-free machine $M$ 
is defined to be the weight of its domain and is denoted $\mathtt{wgt}(M)$.
Without loss of generality we assume that  $\mathtt{wgt}(U)<2^{-2}$.

Prefix-free machines are most often built in terms of {\em request sets}.
A request set $L$ is a set of tuples $\langle \rho, \ell\rangle$ where $\rho$ is a string
and $\ell$ is a positive integer. A `request' $\langle \rho, \ell\rangle$
represents the intention of describing $\rho$
with a string of length $\ell$. We
define the {\em weight of the request} $\langle \rho, \ell\rangle$ to be $2^{-\ell}$.
We say that $L$ is a {\em bounded request set}
if  the sum of the weights of the requests in $L$ is less than 1.
This sum is the {\em weight of the request set $L$} and is denoted by $\mathtt{wgt}(L)$.
The Kraft-Chaitin theorem (see e.g.\ \cite[Section 2.6]{rodenisbook}) 
says that for every bounded request set $L$
which is c.e., there exists a prefix-free 
machine $M$ such that for each $\langle \rho, \ell\rangle\in L$
there exists a string $\tau$ of length $\ell$ such that $M(\tau)=\rho$.
We freely use this method of construction without explicit reference to the
Kraft-Chaitin theorem.
A real number $0\leq r<1$ is called computably enumerable (c.e.)  if it is the 
limit of a non-decreasing computable sequence of rational
numbers.  The binary strings are ordered first by length and then lexicographically.

\subsection{Constructions in computability theory}
This brief discussion is relevant to the constructions
of Sections \ref{subse:profthA} and \ref{subse:profthB}
and is likely to be handy to a reader who is not expert in 
such arguments.
Constructions in computability theory typically take place in stages and 
involve various parameters.
Given a parameter, we use the suffix `$[s]$' to denote
the value of a parameter at the end of stage $s$. In the particular case of some sets
$A, B, D, \emptyset'$ that are enumerated in the course of a construction, we simplify this notation
by making `$s$' a subscript, thus obtaining $A_s, B_s, D_s, \emptyset'_s$
respectively. 
Parameters may have different values at different stages. 
Some parameters are defined in terms of 
the {\em given objects}, for example a fixed universal Turing machine (which is
not in our control) or a given set that is mentioned in the hypothesis of the
theorem that we want to prove. In the case of the construction
of Section \ref{subse:profthA}, the
set $A$ and the universal machine $U$ (along with the Kolmogorov function
$K=K_U)$ are such parameters. We call these {\em parameters of the first type}.
Some parameters are defined in terms of
the objects that we construct, like a machine or a set. 
In the case of the construction
of Section \ref{subse:profthA},
machines $M, N_i$ and the set $B$ are such parameters.
We call these {\em parameters of the second type}.
Most constructions in computability theory are `recursive', in the sense that 
each stage of the construction is defined in terms of
the values of the parameters at the previous stages. 
Usually, we only need to refer to the values of the parameters at the present stage
or the previous stage.
The general rule is that
at each stage of the construction we refer to the values that the parameters of first type have 
at this very stage, while we refer to the values that the parameters of second type
have at (the end of) the previous stage. We follow this standard convention since the values of the
parameters of second type at stage $s$ are only determined at the end of stage $s$.
This rule of thumb is helpful in understanding the formal description of the constructions of Sections
 \ref{subse:profthA} and \ref{subse:profthB}.

\subsection{Coding}\label{subse:coding}
The coding of $\emptyset'$ into $B$ will be implemented through
a system of movable markers
$m_n, n\in\Nat$, where $m_n$ represents 
position in the characteristic sequence of $B$ in which we code the information
of whether $n\in \emptyset'$. Hence we may call $m_n$ 
the $B$-code of the possible event that consists of the enumeration of 
$n$ into $\emptyset'$. 
The movement of the markers as well as the 
computable enumeration of $B$ will take place in the stages of the enumeration of $\emptyset'$.
In particular the value of $m_n$ at stage $s$ is denoted by $m_n[s]$. It is possible that
$m_n[s]$ is undefined (in symbols, $m_n[s]\un$) 
for some $n,s\in\Nat$. The movement of the markers satisfies the 
following standard properties:
\begin{enumerate}
\item[(i)] {\em Monotonicity on stages:} if $m_n[s]\de, m_n[s+1]\de$  then 
$m_n[s]\leq m_n[s+1]$;
\item[(ii)] {\em Monotonicity on indices:}
if $m_n[s]\de, m_{n+1}[s]\de$ then 
$m_n[s]< m_{n+1}[s]$;
\item[(iii)] {\em Consistency:} if $m_n[s]\de, m_n[t]\de$, $m_n[s]\neq m_n[t]$ and $s<t$, 
 then $m_n[s]\in B$; 
\item[(iv)] {\em Convergence:} $\forall n\ \exists t,k\ \forall s\ (s>t\Rightarrow m_n[s]\de=k)$;
\item[(v)] {\em Coding:} $\forall n\ (n\in\emptyset' \iff m_n\in B)$ where $m_n=\lim_s m_n[s]$.
\end{enumerate}
Given a system of markers $(m_n)$ with the above properties,
we can compute $\emptyset'$ given $B$ as follows. In order to
decide if $n\in\emptyset'$, by clause (iii) we may use $B$ in order
to find a stage $s$ such that either $n\in\emptyset'_s$ or
$m_n[s]\not\in B$. In the latter case we know by (v) that
$n\not\in\emptyset'$.

The essence of our method lies on the specific rules that determine the movement
of the markers $m_i$. Intuitively, in order to maintain
$B\leq_K A$ the markers are forced to move many times. Their convergence
is a consequence of the failure to construct a machine
demonstrating that $A$ is $K$-trivial.
Section \ref{subse:profthA} contains the formal argument.

It turns out that
this type of sparse coding may be `permitted' by any finite number of given c.e.\ sets
that are not $K$-trivial. In particular, with some additional effort we can do the same
coding into $B$ while keeping its initial segment  complexity
below any two
given c.e.\ sets $A, C$ that are not $K$-trivial.
Section \ref{subse:profthB} is devoted to the proof
of this generalized result.

\section{Proof of Theorem \ref{th:lowiscom}}\label{subse:profthA}
Let $A$ be a computably enumerable set which is not $K$-trivial.
For the proof of Theorem \ref{th:lowiscom} it suffices to construct 
a computably enumerable set $B$
such that $B\equiv_T\emptyset'$ and $B\leq_K A$.
This follows from the fact that the c.e.\ $K$-degrees are downward dense,
i.e.\ for each c.e.\ set $X$ such that $\emptyset<_K X$ there exists a c.e.\ set $Y$
such that $\emptyset<_K Y<_K X$; see \cite[Section 5]{Barstris}.

\subsection{Parameters and formal requirements of the construction}\label{subse:deparcon}
In order to make $B$ Turing complete we will use 
a system of markers $(m_i)$ as we  discussed in Section
\ref{subse:coding}.
In order to establish $B\leq_K A$ it suffices to construct a prefix-free machine
$M$ such that
\begin{equation}\label{eq:negreqkb}
K_M(B\restr_n)\leq K(A\restr_n)\ \textrm{\ \ \ for all $n$}
\end{equation}
where $K_M$ denotes the prefix-free complexity relative to machine $M$.
Recall that $K$ denotes the prefix-free complexity relative to a fixed universal
prefix-free machine $U$ such that
$\texttt{wgt}(U)<2^{-2}$.

For each marker $m_i$ we enumerate a prefix-free machine
$N_i$ during the construction. The purpose of $N_i$ is
to achieve $\forall n\ (K_{N_i}(A\restr_n)\leq K(n)+c_i)$
for some constant $c_i$.
Since $A$ is not $K$-trivial, this will ultimately fail. However this failure
will help demonstrate that $m_i$
converges: if $m_i$ moves at stage $s+1$ (and all $m_j, j<i$ remain stable), 
the construction refreshes $N_i$ so that $K_{N_i}(A\restr_n)[s+1]\leq K(n)[s+1]+c_i$ holds for the least
$n$ such that $K_{N_i}(A\restr_n)[s]> K(n)[s]+c_i$.
The value of $c_i$ may increase during the construction.
This happens each time some $m_j$, $j<i$ moves. 
Such an event is often described as an `injury' of $m_i$.
In particular, if at some stage $s$
marker $m_k$ moves while $m_j$, $j<k$ remain constant
this causes $m_i$, $i>k$ to be injured,
which has the following consequences:
\begin{itemize}
\item for each $i>k$, markers $m_i$  become undefined and $N_i$ is reset;
\item the values $c_i, i>k$ increase by 1.
\end{itemize}
To `reset' machine $N_i$ means 
to discard all of its computations thus starting to build a new
machine.
Each marker will only be injured finitely many times.
We let $c_i[s]$ denote the 
value of $c_i$ at stage $s$.
At each stage $s$ let $t_i[s]$ be defined as follows:
\[
\parbox{11cm}{$t_i[s]$ is the least number $t$ such that $K_{N_i}(A_{s+1}\restr_t)[s]> K(t)[s+1]+c_i[s]$.}
\]
Each marker $m_i$ has the incentive to move at some stage $s+1$ if it 
observes a set of descriptions of sufficient weight of segments of $A_{s+1}$ 
that are longer than its current position.
This weight is determined by the number (a sort of a `threshold') 
\begin{equation}\label{eq:thresqis}
q_i[s] =2^{-K(t_i[s])[s+1]-c_i[s]}.\end{equation}
The marker $m_i$ requires attention at stage $s+1$ if
$m_i[s]$ is defined, 
$m_i[s]\not\in B_{s}$ and one of the following occurs:
\begin{itemize}
\item[(a)] $i\in \emptyset'[s+1]$;
\item[(b)] $\sum_{m_i[s]<j\leq s} 2^{-K(A\restr_j)[s+1]} \geq q_i[s]$;
\end{itemize}
Note that if $m_i[s]\in B_{s}$ then we must have $i\in \emptyset'_s$.
Hence in this case we do not have any direct reason to move $m_i$ even if (b) holds,
because there will not be any latter stage where we need to enumerate
$m_i[s]$ into $B$ (it is already in it). Of course some $m_j$ with $j<i$
may move at a latter stage, in which case we will need to move $m_i$ too,
but this amounts to a typical finite injury aspect of the construction. 
Alternatively it is clear that we could have set up the
construction so that the condition $m_i[s]\not\in B_{s}$ is not present in the above definition
of $m_i$ `requiring attention'.

For each $i\in\Nat$ we set $c_i[0]=i+3$.
At each stage $s+1$ the machines $N_i$ will be adjusted according to
changes of $K(n)$ for $n<t_i[s]$. This is done by running the 
subroutine (\ref{eq:tsubroureprele}) of the construction in Section
\ref{subse:constr1s}.
A {\em large} number at stage $s+1$ is one that is larger than any number that has been
the value of any parameter in the construction up to stage $s$. Note that
an enumeration of a number $n$ into $B$ only changes the segments
$B\restr_i$ for $i>n$, since $B\restr_i$ consists of the first $i$ bits of $B$, and the
last of these is $B(i-1)$. This is why in the construction below, if we enumerate 
$m_n[s]$ into $B$, we only need to `refresh' the descriptions of $B\restr_k$ 
for $k>m_n[s]$.

\subsection{Intuitive explanation of the dynamics in the construction}\label{sec:intuition} 
Before we give the formal construction and verification, we present some
intuitive and informal comments that illustrate the ideas behind the argument. 
The discussion consists of thee parts: the description of the main conflict
(the coding increases the size of $M$), the simplistic solution to the conflict
(which unfortunately causes the coding procedure to diverge) 
and the final solution that makes all requirements
satisfied. The arguments that we present informally here (especially the third part
of the discussion) correspond to the formal part of the proof in Section \ref{subse:versingt}.

\subsubsection{The main conflict: bounding \texorpdfstring{$M$}{M}
versus coding}\label{sssec:mainconfl}
We use the family of (movable) markers $(m_i)$ in order to ensure that 
$\emptyset'\leq_T B$ as we elaborated in Section \ref{subse:coding}.
On the other hand, we continuously enumerate computations in the machine $M$
according to \eqref{eq:negreqkb}. At each stage these computations ensure
that the initial segment complexity of $B$ (up to a certain length) with respect to $M$
is not greater than the initial segment complexity of $A$. In this way, certain descriptions in
the domain of $U$ (which defines the Kolmogorov function $K=K_U$) induce 
the enumeration of $M$-descriptions of the same length. 

The primary conflict in this argument
is that the coding will cause certain numbers to be enumerated into $B$, and these changes
of the approximation to $B$ will increase the weight of the domain of $M$. This happens because
for every change of the approximation to $B\restr_n$ we need to enumerate an additional description
(corresponding to the new value of $B\restr_n$), possibly of the same length $K(A\restr_n)$
(if the approximation to $A$ has remained the same). This standard conflict is depicted in
Figure \ref{fig:dynam0} and will be present throughout the argument.
Here the solid arrows indicate that enumeration of the codes into $B$ cause the enumeration
of additional weight in the domain of machine $M$.
The dashed arrows between the markers indicate the finite injury effect that occurs amongst them,
which was already indicated in Section \ref{subse:coding}.

A typical situation which illustrates this conflict is the following. 
At some stage $s_0$ we enumerate
an $M$-description of $B\restr_{m_i}[s_0]$ of length $K(A\restr_{m_i})[s_0]$ 
according to \eqref{eq:negreqkb}.
Let $k=m_i[s_0]$ and assume that $K(A\restr_{k})[s_0]=K(A\restr_{k})$.
At some latter stage $s_1$, the number $i-1$ enters $\emptyset'$ 
and we are forced to enumerate $m_{i-1}[s_1]<k$ into $B$. Subsequently,
$i-2$ enters $\emptyset'$ at some stage $s_2$ 
and provokes the enumeration of $m_{i-2}[s_2]<k$ into $B$. And so on, until
$m_0[s_i]<k$ is enumerated into $B$ at some stage $s_i$. During this `cascade' of enumerations, 
the construction will be enumerating
descriptions of the current approximation to $B\restr_{k}$. Hence the construction
will enumerate at least $k$ descriptions of the same length $K(A\restr_{k})$.
If $K(A\restr_{k})=3$ and $k=2^{4}$ then clearly it is not possible to ensure that the weight of 
$M$ is bounded.

\subsubsection{A step to the solution: additional movement of the markers}\label{subs:steptosol}
We deal with the problem of bounding the weight of $M$ 
by moving the markers $(m_i)$, even before
their index (or a smaller index) is enumerated in $\emptyset'$. Of course, this movement
will obey the rules that we set out in Section \ref{subse:coding}.
We will show that by setting appropriate movement rules for the markers,
we can argue that the weight of $M$ is bounded. Figure \ref{fig:dynam1} illustrates
the dynamics of this construction (which is determined below). The features of the crude construction of
Section \ref{sssec:mainconfl} continue to apply here: 
computations in $U$ provoke the enumeration of computations
in $M$ and the activity of the markers trigger the enumeration
of additional $M$-descriptions (while `injuring'  the markers with larger index). Note that
the additional movement of the markers that we enforce in the current form of the construction (see below) induce
additional enumerations of computations in $M$. 
 Figure \ref{fig:dynam1} also features arrows from $M$ to the markers: these illustrate that the enumeration of
 $M$-computations sometimes triggers the movement of the markers. In the following we explain exactly how this
 construction works and why it ensures that the weight of $M$ is bounded. 
 
\begin{figure}
 \scalebox{0.8}{
\begin{tikzpicture}[
main/.style={circle, minimum size=6mm, very thick,
draw=red!50!black!50, top color=white, bottom color=red!50!black!20, font=\small},
U/.style={circle, very thick,
draw=white!50!black!50, top color=white, bottom color=white!50!black!20, font=\small},
Nnode/.style={rectangle,  minimum size=6mm, very thick,
draw=green!50!black!50, top color=white, bottom color=green!50!black!20, font=\small },
tnode/.style={rectangle, rounded corners, very thick,
draw=yellow!50!black!50, top color=white, bottom color=yellow!50!black!20, 
minimum width =10mm, font=\small},
skip loop/.style={to path={-- ++(0,#1) -| (\tikztotarget)}}]
\node (M) [main] at (0,0) {\parbox{1.7cm}{Machine $M$}};
\node(U)[U] at (-2,0) {$U$};
\draw [->] (U) -- (M);
\node (m0) [tnode] at (5,2) {$m_0$};
\node (m1) [tnode] at (5,1) {$m_1$};
\node (m2) [tnode] at (5,0) {$m_2$};
\node (dots) at (5,-1) {$\cdots$};
\node (dots2) at (3,-1) {$\cdots$};
\draw [<-] (M) -- (m0.west);
\draw [<-] (M) -- (m1.west);
\draw [<-] (M) -- (m2.west);
\draw [->, dotted] (m0) -- (m1);
\draw [->, dotted] (m1) -- (m2);
\end{tikzpicture}
}
\caption{{\textrm Machine $U$ adds computations in $M$, while the activity of the markers 
causes additional (`copies' of the previous) computations to be added in $M$.}}
\label{fig:dynam0}
\end{figure}
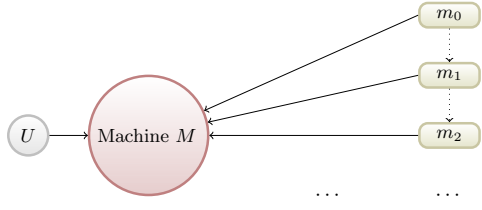

We describe a rule for moving the markers $m_i$ which guarantees that
the weight of $M$ is bounded. 
Let $w_i[s]=\sum_{n>i} 2^{-K(n)[s]}$.
The rule is that marker $m_i$ will move at stage $s$ if $w_i[s]>2^{-r_i[s]-i}$,
where $r_i[s]$ is the number of times that it has moved prior to stage $s$. Of course
we also obey the movement rules that were set out in Section \ref{subse:coding}
(i.e.\ it also moves if $i$ is newly enumerated in $\emptyset'$ or if some $m_j$ with
$j<i$ moves at stage $s$). 
In a standard fashion, we will only enumerate an $M$-description for some $B\restr_n$ if all
markers that occupy positions $< n$ `appear to be stable', namely they have not moved since 
the last stage.

We can argue that in this case the weight of $M$ is
bounded is as follows. Every $M$-description 
(of an approximation to a segment of $B$) corresponds to a $U$-description
(of an approximation to a segment of $A$), where $U$ is the universal
prefix-free machine. Indeed, every $M$-description $\tau$ (describing an initial segment of the current
approximation to $B$) 
is issued according to a certain
$U$-description $\sigma$ (describing an initial segment 
of the same length of the current approximation to $A$).
In this case we say that $\sigma$ is {\em used} by $\tau$.
If $\sigma$ has already been used by $\tau$ and it is later used by a different string
$\tau'$, then we say that $\sigma$ has been {\em reused}.
As illustrated in the above `cascade' example, a $U$-description may be {\em used} by
several $M$-descriptions. In other words, the
correspondence between the domains of $U$ and $M$ 
is not necessarily one-to-one. 
However every $M$-description always corresponds to a $U$-description of equal length.
We will use the weight of $U$ in order to bound the weight of $M$
as follows. Let $S_0$ denote the 
strings in $U$ that are {\em used} by at least one description in $M$ during the construction. 
Clearly $S_0$ is a subset of the domain of $U$, so
$\texttt{wgt}(S_0)< 2^{-k}$.
Also let $S_1$ be the
set of $U$-descriptions that are used by at least two $M$-descriptions. More generally, let $S_k$ be
the set of $U$-descriptions that are used by at least $k+1$ descriptions in $M$.
Then $S_{k+1}\subseteq S_k$, so this family of sets can be illustrated as in Figure \ref{fig:treq}.
Note that if a string $\sigma$ in $S_k$ enters $S_{k+1}$ then there is a unique marker $m_i$ that `causes' this
change. Indeed, $\sigma$ is {\em used} a one more time, which means that the approximation
to the segment $B\restr_n$
(where $n$ is the length of the segment of the current approximation to $A$ that $\sigma$ describes)
changes, due to the enumeration of (the current value of) a marker into $B$. Let $m_i$ be this marker
(if there are more than one markers with this property, we choose the one with the least index). We say that
the entry of $\sigma$ into $S_{k+1}$ is due to the movement of $m_i$.

According to the correspondence between the domains of $U$ and $M$ that we discussed
above,  we can use 
\begin{equation*}
\mathtt{wgt}(M)\leq \sum_k \mathtt{wgt}(S_k).
\end{equation*}
 to bound the weight of $M$. Note that each description in $S_k$ is
counted $k+1$ times in this sum as it belongs to all $S_i, i\leq k$. So it suffices to show that
$\texttt{wgt}(S_k)< 2^{-k}$ for each $k$. Since 
$\texttt{wgt}(S_0)< 2^{-2}$ we also have $\texttt{wgt}(S_1)< 2^{-2}$. Let $k>1$.
Every entry of a string into $S_k$ must have occurred due to the movement of a marker $m_x$.
Moreover, it must have followed the entry of the string into $S_{k-1}$, which in turn must have occurred due to
the movement of a marker $m_y$ with $y>x$. Inductively, every string that enters $S_k$ must be one of the strings
that was  previously enumerated in $S_1$ due to the movement of a marker $m_z$ with $z\geq k-1$.
Let $S_k^z$ be the set of $U$-descriptions in $S_k$ that enter $S_1$ due to the movement
of marker $m_z$. 
Then $\texttt{wgt}(S_k)\leq \sum_{z\geq k-1} \texttt{wgt}(S^z_k)$ for each $k>1$.
Hence it remains to show that $\texttt{wgt}(S^z_1)< 2^{-z-2}$ for each $z$. This follows by the
way we defined the movement of each $m_z$. The $i$th time it moves it is responsible for
new $M$-descriptions of weight at most $2^{-z-i-3}$. So overall
$\texttt{wgt}(S^z_k)$ is bounded by $\sum_i 2^{-z-i-3}= 2^{-z-i-2}$.\footnote{There is a 
more direct way to argue that the weight of $M$ is bounded, by assigning
the additional $M$-descriptions that are issued to the individual markers that caused the relevant changes to the approximation
to $B$. However this argument does not apply to the full construction. The argument we presented here
will be used largely intact in the proof of Theorems \ref{th:lowiscom} and \ref{th:lowismulticom}.}

\begin{figure}
\scalebox{0.8}{
\begin{tikzpicture}[
main/.style={circle, minimum size=6mm, very thick,
draw=red!50!black!50, top color=white, bottom color=red!50!black!20, font=\small},
U/.style={circle, very thick,
draw=white!50!black!50, top color=white, bottom color=white!50!black!20, font=\small},
Nnode/.style={rectangle,  minimum size=6mm, very thick,
draw=green!50!black!50, top color=white, bottom color=green!50!black!20, font=\small },
tnode/.style={rectangle, rounded corners, very thick,
draw=yellow!50!black!50, top color=white, bottom color=yellow!50!black!20, 
minimum width =10mm, font=\small}]
\node (M) [main] at (0,0) {\parbox{1.7cm}{Machine $M$}};
\node(U)[U] at (-2,0) {$U$};
\draw [->] (U) -- (M);
\node (m0) [tnode] at (5,2) {$m_0$-movement};
\node (m1) [tnode] at (5,1) {$m_1$-movement};
\node (m2) [tnode] at (5,0) {$m_2$-movement};
\node (dots) at (5,-1) {$\cdots$};
\node (dots2) at (3,-1) {$\cdots$};
\draw [->] (M) to [in=174]  (m0);
\draw [->, dashed] (m0) to [out=190, in=35]  (M);
\draw [->] (M) to [out=20, in=180] (m1);
\draw [->, dashed] (m1) to [out=193, in=10] (M);
\draw [->] (M) -- (m2.west);
\draw [->, dashed] (m2) to [out=190, in=-10] (M);
\draw [->, dotted] (m0) to (m1);
\draw [->, dotted] (m1) -- (m2);
\end{tikzpicture}
}
\caption{{\textrm Bounding the weight of $M$ requires several movements of the markers. 
In this scheme, increase of the
domain of $M$ provokes the movement of the markers, which in turn
induces additional enumerations of computations into $M$.}}
\label{fig:dynam1}
\end{figure}
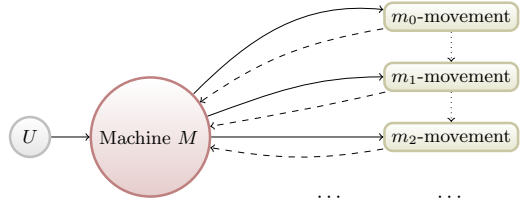

\subsubsection{Ensuring that the markers eventually halt}
The construction of Section \ref{sssec:mainconfl} is based on the rule 
`we move a marker $m_i$ at stage $s$ if the weight of the $M$ descriptions of
$B$ that we will be called to re-describe if $m_i$ is enumerated in $B$, is {\em large}'. In this case
we interpreted `large' as `more than $2^{-r_i[s]-i}$ where $r_i[s]$ is the number of times that $m_i$ has moved
by stage $s$. We refer to $2^{-r_i[s]-i}$ as the {\em threshold for the movement of} $m_i$.
Although this rule allows us to argue that the weight of $M$ is bounded 
(which was the main conflict that was described in Section \ref{sssec:mainconfl}), it is not hard to see that
it causes some markers to move indefinitely.
Clearly this is not in line with the requirements that we set out in
Section \ref{subse:coding} (which are sufficient for 
deducing that $\emptyset'\leq_T B$), so we need to tune
the movement rules for the markers in order to ensure that all requirements are satisfied.
This adjustment will take into account the so-far-unused hypothesis that $A$ 
is not $K$-trivial. 

The idea here is to tie the movement of each marker $m_i$ with the computations enumerated
in an auxiliary machine $N_i$ (constructed by us) which attempts to show that $A$ is $K$-trivial.
The enumerations into $N_i$ will take place at stages where $m_i$ moves and will 
keep the weight of $N_i$ bounded.
We need to define the threshold $q_i$ for the movement of $m_i$ in such a way that
indefinite movement of $m_i$ implies that $N_i$ succeeds its purpose, 
i.e.\ $\forall n, K_{N_i}(A\restr_n)\leq K(n)+c_i$
for some constant $c_i$. 
The threshold $q_i$ is defined in such a way that the enumerations of computations into $M,U, N_i$ are
connected quantitatively. This is essential as the bounds on the weight of $N_i, M$ are eventually reduced to
a bound on the weight of $U$.
The formal definitions of the parameters were given in the beginning of
Section \ref{subse:profthA} and the formal construction is given in Section \ref{subse:constr1s}.

Figure \ref{fig:dynam2} illustrates the dynamics of this refined argument. The features that were discussed in
Sections \ref{sssec:mainconfl} and \ref{subs:steptosol} continue to be present here. 
In addition, the cycle between the growth of $M$ and the movement of $m_i$ fuels the growth of 
an auxiliary machine $N_i$
In particular, the movement of $m_i$ not only adds to the weight of $M$ but also 
triggers the enumeration of additional computations into $N_i$. The growth of $N_i$ threatens to show
that $A$ is $K$-trivial, so it cannot continue indefinitely (and the same holds for the movement of $m_i$). Moreover,
enumeration into $N_i$ causes the `injury' of $N_j$ for all $j>i$. This means that in such cases we initialise $N_j$,
deleting all of its computations and start with a new copy of it. This does not cause any problem to the verification of
the argument, which is done inductively.

Let us conclude this informal discussion with a summary of the mechanics that is illustrated in Figure
\ref{fig:dynam2} and the way it relates to the formal definitions of the parameters $N_i, q_i, m_i, t_i, c_i$.
Every time $m_i$ moves, it enumerates $N_i$ descriptions
(threatening to show that $A$ is $K$-trivial, if these movements happen
indefinitely).
But to keep the weight of $N_i$ bounded, we need to count it against the weight of
the universal machine (via the parameter $c_i$). This is why the condition for movement is the inequality
(b), which is based on the threshold $q_i$ which in turn is defined in terms of $t_i$. 
Actually this is only one of the two reasons, the other being
the use of (b) in bounding the weight of $M$ (see below).
The moment that the sum of descriptions of initial segments of $A$
(for larger lengths than $m_i$) {\em hits} $q_i$, we may move $m_i$ 
and add weight $q_i$
to $N_i$. This is because the opponent (the universal machine $U$) showed
us weight $q_i$
(or even more) in descriptions of certain lengths. The next time that we
enumerate in $N_i$,
we justify the increase in the weight of $N_i$ with {\em different} descriptions of $U$
(indeed, descriptions that describe strings of different lengths, because each
time we move $m_i$ to large values).

This is just one side of the picture. The other side is the dynamics regarding
the enumeration of machine $M$.  Here, intuitively, the more we move the markers,
the more we can save on the weight of $M$ (and the more we add to the
machines $N_i$) as we illustrated in Section \ref{subs:steptosol}.
So it is a rather delicate balance that makes the construction work. This is
crystallised by the inequality (b). Choosing the suitable threshold $q_i$ for triggering
movement of $m_i$ is a crucial part of the argument, as it provides a quantitative connection
between the movement of marker $m_i$ with the enumeration of additional computations in
$M$ and in $N_i$. In the verification of the construction, the fact that machines $N_i$ have bounded
weight (as long as they are not `injured') will be immediate. Then an argument along the lines of the
argument of Section  \ref{subs:steptosol} shows that the weight of $M$ is bounded. Finally, the convergence
of the markers $m_i$ follows inductively, using  $N_i$ and the hypothesis that $A$ is not $K$-trivial.

\subsection{Construction of \texorpdfstring{$B, M, N_i$}{B,M,N}}\label{subse:constr1s}
At stage 0 place $m_0$ on $1$.
At stage $s+1$ run subroutine (\ref{eq:tsubroureprele}).
\begin{equation}\label{eq:tsubroureprele}
\parbox{10.5cm}{For each $i\leq s$ and each $k< t_i[s]$, 
if $K(k)[s+1]<K(k)[s]$ then enumerate an $N_i$-description
of $A_{s+1}\restr_k$ of length $K(k)[s+1]+c_i[s]$.}
\end{equation}
Let $z$ be the least number $<s$
such that $K_M(B_{s}\restr_z)[s]>K(A\restr_z)[s+1]$.
If none of the currently defined markers requires attention,
let $n$ be the least number such that $m_n[s]$ is undefined, and 

\begin{itemize}
\item if $n<z$ place $m_{n}$ on the least {\em large} number;
\item  if $z\leq n$ enumerate an $M$-description of $B_{s}\restr_z$ of length
$K(A\restr_z)[s+1]$;
\item end this stage.
\end{itemize}
Otherwise let $n$ be the least number such that $m_n$ requires attention,
put $m_n[s]$ into $B$,
let $m_n[s+1]$ be a {\em large} number and for each $k<s$ such that $k>m_n[s]$ and
$K_M(B\restr_k)[s]\leq K(A\restr_k)[s+1]$
enumerate an $M$-description of $B_{s+1}\restr_k$ of length $K(A\restr_k)[s+1]$.
Moreover 
for each
$j>n$ declare $m_j[s+1]$ undefined, reset $N_j$ and
set $c_j[s+1]=c_j[s]+1$.
If clause (b) of Section \ref{subse:deparcon} applies,
\begin{equation}\label{eq:ifballpenumNax}
\parbox{11.3cm}{enumerate an $N_n$-description of $A_{s+1}\restr_{t_n[s]}$ of
length $K(t_n[s])[s+1]+c_n[s]$.}
\end{equation}
End this stage.

\subsection{Verification}\label{subse:versingt}
Before we start with the main part of the verification, we make
two preliminary observations that follow directly from the construction. 
The first one concerns the relationship between the values of parameters $t_i$ and $m_i$
during the stages of the construction. When $m_i$ is first defined
at some stage $s$ it takes a {\em large} value so $t_i[s]< m_i[s]$.
Moreover $t_i$ can only increase when $N_i$ computations are enumerated
on strings of length $t_i$,
which happens only when
$m_i$ moves. Also if $A\restr_{t_i}$ changes, by the definition of $t_i$ 
(since $A$ is c.e.\ and $N_i$ is built by us) 
it follows that $t_i$ decreases as soon as $m_i$ moves. Hence by induction we have (\ref{eq:timirele}).
\begin{equation}\label{eq:timirele}
\parbox{8cm}{For all $i,s$, if $m_i[s]$ is defined then $t_i[s]< m_i[s]$.}
\end{equation}
The second observation is a conditional monotonicity on the values of $t_i$ during the stages.
If $K(k)$ decreases at some stage $s+1$
for some $k<t_i[s]$, subroutine
(\ref{eq:tsubroureprele}) will ensure that $K_{N_i}(A\restr_k)[s+1]\leq K(k)[s+1]+c_i[s]$.
Hence $t_i$ may only decrease at
$s+1$ if $A_{s+1}\restr_{t_i[s]}\neq A_s\restr_{t_i[s]}$. 
\begin{equation}\label{eq:tmovmonle}
\parbox{10cm}{If  $A_s\restr_{t_i[s]}=A_{s+1}\restr_{t_i[s]}$ then $t_i[s]\leq t_i[s+1]$.
}
\end{equation}
We are now ready to proceed with the first step of the verification, which is to show
that for each $i$ there is a machine $N_i$ as
prescribed in the construction. Recall that the construction may reset $N_i$. This means
that for each $i$ we have many versions of $N_i$. A new version of $N_i$ is placed when the latest one
is reset. In that case all the previous versions of $N_i$ are no more relevant in the rest of the construction
(in particular, they do not change anymore). When we refer to $N_i$ we refer to an arbitrary version of it
and the interval of stages from its introduction until (if ever) it is reset (before its introduction it is empty and after
it is reset it remains constant).
\begin{lem}\label{le:nibound}
For each $i$ the weight of the requests in $N_i$ is bounded.
\end{lem}
\begin{proof}
We consider an arbitrary version of $N_i$ and it suffices to prove the lemma
for the interval of stages $[b_0, b_1)$ where $b_0$ is the stage where it was introduced
and $b_1$ is the stage when it was reset (so $b_1$ may be $\infty$).
By the construction, all $m_j[s], j<i$ and $c_i[s]$ remain stable during
the stages $s$ in $[b_0, b_1)$. So in the following we may refer to 
$c_i[s]$ by $c_i$.

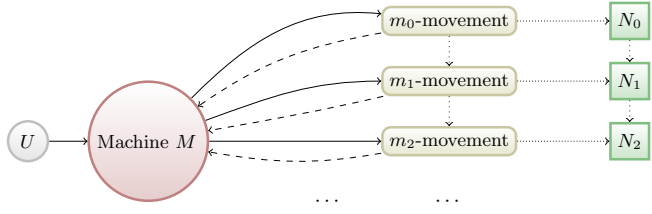
\begin{figure}
\scalebox{0.8}{
\begin{tikzpicture}[
main/.style={circle, minimum size=6mm, very thick,
draw=red!50!black!50, top color=white, bottom color=red!50!black!20, font=\small},
U/.style={circle, very thick,
draw=white!50!black!50, top color=white, bottom color=white!50!black!20, font=\small},
Nnode/.style={rectangle,  minimum size=6mm, very thick,
draw=green!50!black!50, top color=white, bottom color=green!50!black!20, font=\small },
tnode/.style={rectangle, rounded corners, very thick,
draw=yellow!50!black!50, top color=white, bottom color=yellow!50!black!20, 
 font=\small}]
\node (M) [main] at (0,0) {\parbox{1.7cm}{Machine $M$}};
\node(U)[U] at (-2,0) {$U$};
\draw [->] (U) -- (M);
\node (m0) [tnode] at (5,2) {$m_0$-movement};
\node (N0) [Nnode] at (8,2) {$N_0$};
\node (m1) [tnode] at (5,1) {$m_1$-movement};
\node (N1) [Nnode] at (8,1) {$N_1$};
\node (m2) [tnode] at (5,0) {$m_2$-movement};
\node (N2) [Nnode] at (8,0) {$N_2$};
\node (dots) at (5,-1) {$\cdots$};
\node (dots2) at (3,-1) {$\cdots$};
\draw [->] (M) to [in=174]  (m0);
\draw [->, dashed] (m0) to [out=190, in=35]  (M);
\draw [->] (M) to [out=20, in=180] (m1);
\draw [->, dashed] (m1) to [out=193, in=10] (M);
\draw [->] (M) -- (m2.west);
\draw [->, dashed] (m2) to [out=190, in=-10] (M);
\draw [->, densely dotted] (m2) to  (N2);
\draw [->, densely dotted] (m1) to  (N1);
\draw [->, densely dotted] (m0) to  (N0);
\draw [->, dotted] (m0) to (m1);
\draw [->, dotted] (N0) to (N1);
\draw [->, dotted] (N1) to (N2);
\draw [->, dotted] (m1) -- (m2);
\end{tikzpicture}}
\caption{{\textrm Auxiliary machines $N_i$ are fuelled by the cycles between
the enumeration of $M$-computations and movement of the markers. They ensure that if $A$ is not $K$-trivial
these cycles have to stop, which implies that the markers eventually halt.}}
\label{fig:dynam2}
\end{figure}

A request is enumerated into $N_i$ either by
by subroutine (\ref{eq:tsubroureprele}) or due to the movement of a marker $m_i$.
We will bound each part of the $N_i$ requests separately and then add the bounds.
First, we consider the 
requests that are enumerated
by subroutine (\ref{eq:tsubroureprele}).
Each such request is associated with a unique pair $(k,s)$ such that
$K(k)[s+1]<K(k)[s]$. Moreover such a request has weight 
$K(k)[s+1]+c_i$. It follows that
the total weight of these requests is bounded by $2^{-c_i}\cdot \texttt{wgt}(U)$, 
which is at most $2^{-2}$.

The only other way that an enumeration into $N_i$ may be requested is when
a marker $m_i$ requires attention at some stage $s+1$. 
Recall from Section \ref{subse:deparcon} the conditions that need to be met in order for $m_i$ 
to require attention at stage $s+1$, and in particular clause (b).
It follows that in this case  
the marker moves to a {\em large} value  and
the weight of the request is
$q_i[s]\leq \sum_{m_i[s]<j\leq s} 2^{-K(A\restr_j)[s+1]}$.  Let $(s_j)$ be the sequence
of stages in $[b_0, b_1)$ where $m_i$ moves.
Then the weight of the requests that are enumerated in $N_i$ in this way (via the movement of $m_i$)
is bounded by
\[
\sum_j \Big (\sum_{m_i[s_j]<j\leq m_i[s_{j+1}]} 2^{-K(A\restr_j)[s_j]}\Big ) \leq \texttt{wgt}(U).
\]
Hence the weight of the requests that are enumerated in $N_i$ 
in the latter manner (i.e.\ via the movement of $m_i$) is bounded by
$2^{-2}$. Since we established the same bound for the weight of the requests that
are enumerated in $N_i$ via the first manner 
(i.e.\ via  (\ref{eq:tsubroureprele})) it follows that
 $\mathtt{wgt}(N_i)\leq 2^{-2}+2^{-2}=2^{-1}$.
\end{proof}

The following lemma is essential in showing that $\emptyset'\leq_T B$.
The proof of it,
uses the fact that each $N_i$ is a prefix-free machine,
which was established in Lemma \ref{le:nibound}.
\begin{lem}\label{le:miconva}
For each $i$, marker $m_i$ is defined, injured only finitely many times and reaches a limit;
\end{lem}
\begin{proof}
We argue by induction on $i$.
In order to conclude the induction step and the proof of this lemma, it suffices
to show that $m_{i+1}$ will reach a limit.
By the induction hypothesis, $m_{i+1}$
stops being injured after stage some $s_0$.
Hence $c_{i+1}$ reaches a limit at $s_0$.
Since $A$ is not $K$-trivial there is some least $n$ such that
$K_{N_{i+1}}(A\restr_n)>K(n)+c_{i+1}$. Let $s_1>s_0$ be a stage where 
the approximations to $A\restr_n$ and $K(j)$, $j\leq n$ 
have settled. If marker $m_{i+1}$ moved after stage $s_1$ 
the construction would enumerate an $N_{i+1}$-description of $A\restr_n$
of length $K(n)+c_{i+1}[s_1]$ which contradicts the choice of $n$.
Hence $m_{i+1}$ reaches a limit by stage $s_1$ and this 
concludes the induction step and the proof.
\end{proof}

We define $(S_i)$ exactly as in the discussion of Section \ref{subs:steptosol}.
Let $S_0$ be the set of strings in the domain of $U$ 
that are used at least one time. More generally for each $k\geq 0$ we let $S_k$ be the set of 
descriptions in the domain of $U$ which 
are used at least $k+1$ times. Note that $S_{i+1}\subseteq S_i$ for
each $i$. According to the 
correspondence between the domains of $U$ and $M$, a string $\sigma$
in the domain of $U$ that is used $k$ times incurs weight $k\cdot 2^{-|\sigma|}$ 
to the domain
of $M$. Hence (\ref{eq:Mwbou}).
\begin{equation}\label{eq:Mwbou}
\mathtt{wgt}(M)\leq \sum_k \mathtt{wgt}(S_k).
\end{equation}
Note that in the above sum each description in $S_k$ is counted $k+1$ times,
since it is also a member of $S_j$ for $j<k$.
A $U$-description $\sigma$ is called {\em active}
at stage $s$ if $U(\sigma)[s]\subseteq A_s$.
By the construction, all descriptions that enter $S_1$ at
some stage $s$ are active at that stage. More generally, at any given stage $s$, only
strings that are active at stage $s$ may move
from $S_k$ to $S_{k+1}$ at stage $s$. 

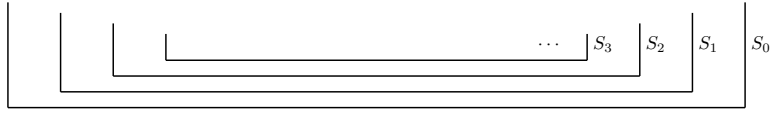
\begin{figure}
 \scalebox{0.7}{
\begin{tikzpicture}
\draw[black, thick] (-7,0)  -- (7,0);
\draw[black, thick] (-7,0)  -- (-7,2);
\draw[black, thick] (7,0)  -- (7,2);
\draw[black, thick] (-6,0.3)  -- (6,0.3);
\draw[black, thick] (-6,0.3)  -- (-6,1.8);
\draw[black, thick] (6,0.3)  -- (6,1.8);
\draw[black, thick] (-5,0.6)  -- (5,0.6);
\draw[black, thick] (-5,0.6)  -- (-5,1.6);
\draw[black, thick] (5,0.6)  -- (5,1.6);
\draw[black, thick] (-4,0.9)  -- (4,0.9);
\draw[black, thick] (-4,0.9)  -- (-4,1.4);
\draw[black, thick] (4,0.9)  -- (4,1.4);
\node  at (7.3,1.2) {$S_0$};
\node  at (6.3,1.2) {$S_1$};
\node  at (5.3,1.2) {$S_2$};
\node  at (4.3,1.2) {$S_3$};
\node  at (3.3,1.2) {$\cdots$};
\end{tikzpicture}}
\caption{{\small Infinite nested decanter model.}}
\label{fig:treq}
\end{figure}
The sets $S_k$ may be visualized as the nested containers of the
infinite decanter model of Figure \ref{fig:treq}. As the figure indicates,
descriptions that are currently in container $S_k$ may enter container $S_{k+1}$ while they
continue to be members of $S_k$. In particular, once a description enters a container it will remain in that
container indefinitely.
If at some stage a marker $m_i$ moves (while $m_j, j<i$ remain stable),
some strings of $S_k$ enter $S_{k+1}$ for various $k\in\Nat$.
Indeed, when $m_i$ moves it enumerates its former value into $B$. This action
changes the approximation to $B$, which in turn causes some descriptions to be used
an additional time. By the definition of the sets $(S_k)$, this means 
that some of the strings in some containers enter the next container. In this case we may say 
that these strings
were reused by $m_i$ (since they were used an additional time). 
In order to calculate a suitable upper bound for each
$\mathtt{wgt}(S_k)$ we need Lemma \ref{eq:bfolocoAnsk}.
\begin{lem}\label{eq:bfolocoAnsk}
If during the interval of stages $[s,r]$ 
a marker $m_n$  is not injured 
and $n\not\in\emptyset'_r$
then the weight of the strings that $m_n$ reuses during
this interval which remain active at stage $r$ is at most $2^{-c_n[s]}$.
\end{lem}
\begin{proof}
Note that by the assumption, parameter $c_n$ remains constant throughout the interval
$[s,r]$.
Suppose that $m_n$
moves at stage $x+1\in [s,r]$, after requiring attention. 
Then since the markers always move to {\em large} values
it follows that $m_n$ did not move at stage $x$. 
Recall the definition of when $m_n$ requires attention, which was given in
Section \ref{subse:deparcon} (the clauses (a) and (b)).
Since $m_n$ did not move at stage $x$, it follows 
that it did not require attention at that stage and by clause (b) 
of Section \ref{subse:deparcon} we get 
\begin{equation}\label{eq:mnpayprevdesc}
\sum_{j>m_n[x]} 2^{-K(A \restr_j)[x]} < q_n[x].
\end{equation}
Note that when $m_n$ moves at stage $x+1$,  
the weight of the $U$-descriptions
that it reuses is at most
$\sum_{j>m_n[x]} 2^{-K(A \restr_j)[x]}$ (and not $\sum_{j>m_n[x]} 2^{-K(A \restr_j)[x+1]}$).
This happens because the construction first moves marker $m_n$ and then enumerates
additional computations in $M$. In other words, the descriptions that $m_n$ reuses at $x+1$ 
correspond to $M$-computations that occurred in the previous stages, not the $M$-computations
that may occur by the end of stage $x+1$. 
Hence by \eqref{eq:mnpayprevdesc}, 
the weight of the $U$-descriptions that are reused by $m_n$ at stage $x+1$
are bounded by $q_n[x]$, which is $2^{-K(t_n)[x]-c_n[x]}$.

Now let us consider the overall effect of the movement of $m_n$ 
in the interval of stages $[s,r]$.
If at least one of the descriptions in $U$ that $m_n$ 
reused at some stage  $x+1\in [s,r]$
continues to be active at stage $r$,
then  $A_{x+1}\restr_{m_n[x]+1}=A_r\restr_{m_n[x]+1}$.
By (\ref{eq:timirele}), under the same assumptions this implies
\[
A_{x+1}\restr_{t_n[x]+1}=A_r\restr_{t_n[x]+1}.
\]
By \eqref{eq:ifballpenumNax} of the construction 
(i.e.\ the enumeration of a computation
in $N_i$ upon the movement of a marker), 
since at stage $x+1$ the marker $m_n$ moved, we have
$t_n[x+1]=t_n[x]+1$.
Hence by (\ref{eq:tmovmonle}) we get that
\[
t_n[y]\geq t_n [x+1]>t_n[x] \ \ \ \textrm{for all $y\in [x+1, r]$.}
\]
The above observation along with the bound that we established 
in the previous paragraph
on the weight of the $U$-descriptions that are reused by $m_n$ at  a stage in
$[s, r]$ , imply the following fact.
\begin{equation*}
\parbox{10.7cm}{During the stages in
$[s, r]$ the weight of the descriptions in $U$
 that $m_n$ reuses 
and remain active at stage $r$,  
is each time bounded by $2^{-K(t_n)-c_n}$,
where $t_n$ is larger and larger and $c_n$ remains equal to $c_n[s]$}
\end{equation*}
(while the Kolmogorov function follows its usual approximation). 
Formally, if $y_j, j<k$ are the stages in $[s, r]$ where
marker $m_n$ moves, we have $t_n[y_j]< t_n[y_{j+1}]$ and the
weight of $U$-descriptions that $m_n$ reuses at stage $y_j$ and remain active 
at stage $r$ is at most
$2^{-K(t_n)[y_j]-c_n[s]}$. So the total weight of the $U$-descriptions that $m_n$ uses
during the stages in $[s, r]$ and which remain active at stage $r$ is less than
\[
\sum_i 2^{-K(i)-c_n[s]}.
\]
Since the above sum is bounded by $2^{-c_n[s]}$, this
concludes the proof.
\end{proof}

\begin{lem}\label{le:Mbound}
The weight of the requests that are enumerated in $M$ is finite.
\end{lem}
\begin{proof}
According to the correspondence between the domains of $U$ and $M$ that we discussed,
we can use \eqref{eq:Mwbou} to bound the weight of $M$. Note that each description in $S_k$ is
counted $k+1$ times in this sum as it belongs to all $S_i, i\leq k$. 
Since only strings in the domain of $U$ are used,
$\mathtt{wgt}(S_0)<2^{-2}$.
So it suffices to show that
\begin{equation}\label{eq:boundonsk}
\texttt{wgt}(S_k)< 2^{-k-1} \ \ \textrm{for each $k>0$.}
\end{equation}
Since $S_1\subseteq S_0$, condition \eqref{eq:boundonsk} holds for $k=1$.  
Let $k>1$.
Every entry of a string into $S_{k}$ 
is due to a marker $m_x$ which reused it when it was already in
$S_{k-1}$. Since $k>1$, this string entered $S_{k-1}$ due to another marker 
$m_y$ with $y>x\geq 0$.
Inductively, that string entered $S_1$ due to a marker 
$m_z$ with $z\geq k-1$. 
Fix $z$, and let $S_k^z$ contain the strings in $S_k$ that entered $S_1$
due to marker $m_z$. 
Then $S_k=\cup_{z\geq k-1} S^z_k$
and $S_{k+1}^z \subseteq S_k^z$ 
for each $k>1$. Hence
\[
\texttt{wgt}(S_k)\leq \sum_{z\geq k-1} \texttt{wgt}(S^z_k)\ \ \textrm{for each $k>1$.}
\]
So in order to prove \eqref{eq:boundonsk} for $k>1$ it suffices to show that
\begin{equation}\label{Szkboundf}
\texttt{wgt}(S^z_k)< 2^{-z-2}\ \ \textrm{for each $z\geq 0$.} 
\end{equation}
Let $(s_i)$ be the increasing sequence of 
stages where $m_z$
is injured. Note that at this point we do not assume that
$(s_j)$ is a finite sequence.
We may count the weight of $S_k^z$
by counting the weight of the bunches of descriptions that
moved to $S^z_1$ and then  moved to $S^z_2$
(necessarily by some $m_j$ with $j<z$). 
This is justified because every description that enters
$S_k^z$ must have passed from $S_2^z$ first.

Since the movement of a marker $m_i$ injures all $m_j, j>i$,
the only stages were strings move from 
$S^z_1$ to $S^z_{2}$
are the stages $(s_i)$.
Moreover 
since only active strings move from
$S^z_1$ to $S^z_{2}$ at stage $s_i$, according to Lemma \ref{eq:bfolocoAnsk}
(applied  to the intervals $[s_i+1, s_{i+1}-1]$)
their weight is bounded by $2^{-c_z[s_i-1]}$.
So the weight of the strings that enter
$S^z_{2}$ from $S^z_1$ is bounded above by $\sum_j 2^{-c_z[s_j-1]}$.
Since $c_z[s_{j+1}-1]=c_z[s_j]>c_z[s_{j}-1]$ for
all $j$,  this weight is bounded by $\sum_j 2^{-c_z[0]-j}=2^{-c_z[0]+1}$.
Since $c_z[0]=z+3$ this bound becomes
$2^{-z-2}$, which establishes \eqref{Szkboundf}.
\end{proof}
\noindent
We conclude with the proof that (\ref{eq:negreqkb}) is met.
\begin{lem}\label{le:miconv}
For each $i$ there is an $M$-description of $B\restr_i$ of length $\leq K(A\restr_i)$.
\end{lem}
\begin{proof}
We argue by  induction on $i$.
Suppose that the lemma holds for $i\in\Nat$. Then by Lemma \ref{le:miconva}, 
there is some stage
$s_0$ at which the approximations to $A\restr_{i+1}, B\restr_{i+1}$, 
$K(A\restr_{i+1})$, $K_M(B\restr_i)$ and
$m_i$ have 
settled and $K_M(B\restr_i)[s_0]\leq K(A\restr_i)[s_0]$.
If $K_M(B\restr_{i+1})[s_0]> K(A\restr_{i+1})[s_0]$ the construction
at stage $s_0+1$ will enumerate
an $M$-computation that describes
$B\restr_{i+1}$ with a string of length $K(A\restr_{i+1})$.
\end{proof} 
\noindent
By Lemma \ref{le:miconva} and the construction we get
that the movement of the markers satisfies properties 
(i)-(v) of Section \ref{se:prothlowisa}.
Hence $\emptyset'\leq_T B$. We conclude the proof of Theorem
\ref{th:lowiscom}
by observing that \eqref{eq:negreqkb} is met.
By Lemma \ref{le:miconv} the construction enumerates the required requests in $M$
which ask for a description of $B\restr_i$ 
with a string of length at most $K(A\restr_i)$, for each $i$.
On the other hand Lemma \ref{le:Mbound} establishes that this request set corresponds to
a prefix-free machine, via the Kraft-Chaitin lemma. Hence \eqref{eq:negreqkb} is met,
which concludes the verification of the construction
and the proof of Theorem \ref{th:lowiscom}.

\section{Proof of Theorem \ref{th:lowismulticom}}\label{subse:profthB}
Let $A,D$ be two computably enumerable sets which are not $K$-trivial.
For the proof of Theorem \ref{th:lowismulticom} it suffices to construct 
a computably enumerable set $B$
such that $B\leq_K A$, $B\leq_K D$ and $B\equiv_T\emptyset'$ .
This follows from the downward density of the c.e.\ $K$-degrees
as we discussed in Section \ref{subse:profthA}.
The coding of
$\emptyset'$ into $B$ will be done via the markers $(m_i)$ 
and  
the relations $B\leq_K A$, $B\leq_K D$ will be achieved
with the construction of two prefix-free 
machines $M_a, M_d$ respectively such that
\begin{equation}\label{eq:neduagreqkb}
K_{M_a}(B\restr_n)\leq K(A\restr_n)\  \textrm{\ \ \ and \ \ \ } 
K_{M_d}(B\restr_n)\leq K(D\restr_n)\textrm{\ \ \ for all $n$.}
\end{equation}

\subsection{Merging two constructions}\label{subse:merge}
The basic plan of the construction of $M_a, M_d$ is to
merge a construction for $M_a$ of the type that was
given in Section \ref{subse:profthA} with a construction
for   $M_d$ of the same type.
Note that we will have a single set of 
markers $m_i$ but their movement 
will be stimulated by both requirements in
(\ref{eq:neduagreqkb}).
We will use the same set of constants
$c_i$ for both $A$ and $D$, since their values only
depend on the movement of the markers on $B$.
However for each $i$ we have $N_i^{a}, N_i^{d}$ instead of
$N_i$. 
At each stage $s$ we let $t^a_i[s]$ 
be the least number $x$ such that
$K_{N^a_i}(A_{s+1}\restr_x)[s]> K(x)[s+1]+c_i[s]$
and we let $t^d_i[s]$ 
be the least number $y$ such that
$K_{N^d_i}(D_{s+1}\restr_y)[s]> K(y)[s+1]+c_i[s]$.
For each $i\in\Nat$ we set $c_i[0]=i+4$.
The universal machine $U$ 
and the notion of injury of a marker remains the same. 
In particular, if at some stage $s$
marker $m_k$ moves while $m_j$, $j<k$ remain constant 
this causes $m_i$, $i>k$ to be injured.
This means that $m_i$, $i>k$ become undefined
and the values of $c_i, i>k$ increase by 1.

At each stage $s+1$ the machines $N_i^a, N_i^d$ will be adjusted according to
changes of $K(n)$ for $n<t_i[s]$. This is done by running 
subroutine (\ref{eq:tsubrodouule}) (which is analogous to
\eqref{eq:tsubroureprele} of the argument in Section \ref{subse:profthA}). We define
\[
q^a_i[s]=2^{-K(t^a_i)[s]-c_i[s]}\hspace{0.5cm}\textrm{and}\hspace{0.5cm}
q^d_i[s]=2^{-K(t^d_i)[s]-c_i[s]}.
\]
The thresholds $q^a_i[s]$, $q^d_i[s]$
play a similar role as $q_i[s]$ in the argument of
Section \ref{subse:profthA}.
However since $K(A\restr_n)$ and
$K(D\restr_n)$ may differ for various $n$, the definition of
a marker requiring attention will be modified, as we elaborate in
Section \ref{subse:lackunifsol}.

\subsection{Lack of uniformity and solution}\label{subse:lackunifsol}
The main issue that we have to deal with when we merge
two constructions of the type used in Section \ref{subse:profthA}
which depend on different c.e.\ sets $A,D$
is that the thresholds $q^a_i, q^d_i$ that correspond to some
marker $m_i$ may have different values.
Hence the marker may be motivated to move by $M_a$ but not by $M_d$.
This lack of uniformity has an impact in the calculations of the weight 
of the machines $N_i^a, N_i^d$, which in turn affects a verification along the lines
of Section \ref{subse:versingt}.

The solution to this obstacle
is to use the additional parameters $p^a_i, p^d_i$
which record  
the weight of the $M_a$ or $M_d$ descriptions respectively that
were reissued when only $M_d$ or $M_a$ respectively motivated the movement
of marker $m_i$.
For example, at some stage $s+1$ we may have
$\sum_{m_i[s]<j\leq s} 2^{-K(A\restr_j)[s]} \geq q^a_i[s]$
but this may not hold for $D$ in place of $A$ and 
$q^d_i[s]$ in place of $q^a_i[s]$.
This means that at this stage $M_a$ requires the movement of $m_i$ but
$M_d$ does not.
At such a stage we will move $m_i$ for the sake of $M_a$, 
also enumerating an $N_i^a$-description of $t^a_i[s]$ of length
$K(t^a_i)[s]+c_i[s]$.
However an
enumeration of an $N_i^d$-description of $t^d_i[s]$ of length
$K(t^d_i)[s]+c_i[s]$ is not justified
and will not take place. Instead, we will store the value 
$\sum_{m_i[s]<j\leq s} 2^{-K(D\restr_j)[s]}$
into $p_i^d[s+1]$, which is the weight of the $M_d$ descriptions we need to reuse
due to the movement of $m_i$ at stage $s+1$. At
the next stage the
threshold in the condition for 
the movement of $m_i$ for the sake of $D$ will be
$q^d_i[s+1]-p_i^d[s+1]$. As long as $m_i$ moves for the sake of 
$M_a$ the value of $p_i^d$ will keep on increasing, recording
 the weight of the $M_d$ descriptions that we need to pay
due to the $M_a$-motivated movements of $m_i$.
When $m_i$ moves for the sake of $M_d$, the value of
$p_i^d$ will drop to 0 and the enumeration into
$N_i^d$ will be justified. The same holds symmetrically for
$A$ with $q^a_i$
and $p^a_i$.
With this amendment a combined construction can be verified
along the lines of the argument of Section
\ref{subse:versingt}.

According to the above motivation, we say
that the marker $m_i$ requires attention at stage $s+1$ if
$m_i[s]$ is defined, 
$m_i[s]\not\in B_{s}$ and one of the following occurs:
\begin{itemize}
\item[(a)] $i\in \emptyset'_{s+1}$;
\item[(b)] $\sum_{m_i[s]<j\leq s} 2^{-K(A\restr_j)[s]} \geq q^a_i[s]-p^a_i[s]$;
\item[(c)] $\sum_{m_i[s]<j\leq s} 2^{-K(D\restr_j)[s]} \geq q^d_i[s]-p^d_i[s]$;
\end{itemize}
The condition $m_i[s]\not\in B_{s}$ in the above definition can be justified
as the same condition was justified in the construction of Section 
\ref{subse:constr1s} (see the discussion in the end of Section \ref{subse:deparcon}).
The definition of a {\em large} number is as in the argument of Section \ref{subse:profthA}.
Recall that the parameters $q^a_i, q^d_i$ are defined in terms of the given sets $A,D$ (and the
universal machine $U$) while the parameters 
$p^a_i, p^d_i$ are defined dynamically within the construction. We define 
$p^a_i[0]=p^d_i[0]=0$.

\subsection{Construction of \texorpdfstring{$B, M_a, M_d, N_i^a, N_i^d$}{B,Ma,Md,Na,Nd}}
At stage 0 place $m_0$ on $1$.
At stage $s+1$ run subroutine (\ref{eq:tsubrodouule}) for $(X,x)\in\{(A,a), (D,d)\}$.
\begin{equation}\label{eq:tsubrodouule}
\parbox{11cm}{For each $i\leq s$,  
if $K(k)[s+1]<K(k)[s]$
for some $k< t^x_i[s]$ 
 then enumerate an $N^x_i$-description
of $X_{s+1}\restr_k$ of length $K(k)[s+1]+c_i[s]$.}
\end{equation}
Let $z_a, z_d$ be the least numbers $\leq s$ 
such that 
\[
K_{M_d}(B\restr_{z_a})[s]>K(A\restr_{z_a})[s+1]\hspace{0.5cm}
\textrm{and}\hspace{0.5cm} K_{M_d}(B\restr_{z_d})[s]>K(D\restr_{z_d})[s+1].
\]
If none of the currently defined markers requires attention,
let $n$ be the largest number such that $m_n[s]$ is undefined, 
 and
\begin{itemize}
\item if $n< z_a$ and $n<z_d$, place $m_{n}$ on the least {\em large} number;
\item  otherwise enumerate an $M_a$-description of $B_{s}\restr_{z_a}$ of length
$K(A\restr_{z_a})[s+1]$ and
an $M_d$-description of $B_{s}\restr_{z_d}$ of length
$K(D\restr_{z_d})[s+1]$;
\item end this stage.
\end{itemize}
Otherwise let $n$ be the least number such that $m_n$ requires attention,
put $m_n[s]$ into $B$, let $m_n[s+1]$ be a {\em large} number and for each $k<s$
with $k>m_n[s]$
\begin{itemize}
\item if $K_{M_a}(B\restr_k)[s]\leq K(A\restr_k)[s+1]$
enumerate an $M_a$ description of $B_{s+1}\restr_k$ of length $K(A\restr_k)[s+1]$;
\item if $K_{M_d}(B\restr_k)[s]\leq K(D\restr_k)[s+1]$
enumerate an $M_d$-description of $B_{s+1}\restr_k$ of length $K(D\restr_k)[s+1]$.
\end{itemize}
Moreover for each $j> n$
\begin{itemize}
\item declare $m_j[s+1]$ undefined and reset machines
$N^a_j$, $N^d_j$; 
\item set $c_j[s+1]=c_i[s]+1$ and $p_j^a[s+1]=p_j^d[s+1]=0$. 
\end{itemize}
Finally for $(X,x)\in\{(A,a), (D,d)\}$ consider the action
\begin{equation*}
\parbox{12.2cm}{($\ast$)\ enumerate an $N^x_n$-description of $X\restr_{t^x_n[s]}[s+1]$ of
length $K(t^x_n[s])[s+1]+ c_n[s]$.}
\end{equation*}
and do the following, according to whether clauses (b), (c)
of Section \ref{subse:lackunifsol} hold:
\begin{itemize}
\item If (b), (c) hold, for $(X,x)\in\{(A,a), (D,d)\}$ execute ($\ast$)  
and set $p^x_n[s+1]=0$;
\item otherwise, if (b) holds, execute ($\ast$)
for $(X,x)=(A,a)$ and set $p^a_n[s+1]=0$,
$p_n^d[s+1]=p_n^d[s]+\sum_{m_n[s]<j\leq s} 2^{-K(D\restr_j)[s+1]}$;
\item otherwise, if (a) holds
execute ($\ast$)
for $(X,x)=(D,d)$ and
set $p^d_n[s+1]=0$, $p_n^a[s+1]=p_n^a[s]+
\sum_{m_n[s]<j\leq s} 2^{-K(A\restr_j)[s+1]}$.
\end{itemize}
End this stage.

\subsection{Verification}
As in the verification of Section \ref{subse:profthA}
we have (\ref{eq:timidrele}).
\begin{equation}\label{eq:timidrele}
\parbox{10cm}{For all $i,s$, if $m_i[s]$ is defined then $t^a_i[s]< m_i[s]$
and $t^d_i[s]< m_i[s]$.}
\end{equation}
Moreover the justification of (\ref{eq:tmovmonle})
also applies to (\ref{eq:tmovdmonle}).
\begin{equation}\label{eq:tmovdmonle}
\parbox{10cm}{If  $A_s\restr_{t^a_i[s]}=A_{s+1}\restr_{t^a_i[s]}$ then $t^a_i[s]\leq t^a_i[s+1]$.\\
If  $D_{s}\restr_{t^d_i[s]}=D_{s+1}\restr_{t^d_i[s]}$ then $t^d_i[s]\leq t^d_i[s+1]$.
}
\end{equation}
Next, we show that for each $i$ there are machines $N^a_i, N^d_i$ as
prescribed in the construction.
The proof of this fact is slightly more involved than the
corresponding fact in the argument of Section \ref{subse:profthA}
due to the amendment that was discussed in Section \ref{subse:lackunifsol}.

\begin{lem}\label{le:nibodund}
For each $i$ the weights of the requests in 
$N^a_i$ and $N^d_i$  are bounded.
\end{lem}
\begin{proof}

Let $(X,x)\in\{(A,a), (D,d)\}$.
As in Section \ref{subse:profthA}, each machine $N^x_i$
is valid only as long as $m_i$ is not injured. In this way we have many copies of $N^x_i$ and
it suffices to argue about a fixed version  of it (which is relevant only in an interval of stages where
$m_i$ is not injured).

A request is enumerated into $N^x_i$ either by
by subroutine \eqref{eq:tsubrodouule} or due to the movement of a marker $m_i$.
We will bound each part of the $N_i$ requests separately and then add the two bounds.
First, we consider the 
requests that are enumerated
by subroutine \eqref{eq:tsubrodouule}.
Each such request is associated with a unique pair $(k,s)$ such that
$K(k)[s+1]<K(k)[s]$. Moreover such a request has weight 
$K(k)[s+1]+c_i$. It follows that
the total weight of these requests is bounded by $2^{-c_i}\cdot \texttt{wgt}(U)$, 
which is at most $2^{-2}$.

Let $(s_j)$ be the sequence of stages where
$m_i$ moves, inside an interval $J$ of stages $s$ where $m_i$ is not injured
and $i\not\in\emptyset'_s$. Moreover let $I_j=(m_i[s_j], m_i[s_{j+1}]]$
be the interval that marker $m_i$ crosses 
when it moves at stage $s_{j}$.
For each $j$ let
\begin{equation*} 
x_j=\sum_{n\in I_j} 2^{-K(X\restr_n)[s_j-1]}.
\end{equation*}
If $x=a$ let $(e_x)$ be clause (b)
of Section \ref{subse:lackunifsol} 
and if
$x=d$ let ($e_x$) be clause (c)
of Section \ref{subse:lackunifsol}.
Let $(k^x_j)$ be the monotone sequence of those numbers $k$ such that
at stage $s_k$ marker $m_i$ moves due to clause ($e_x$). 
According to the construction and the way we increase $p^x_i$,
the weight of the $N_i^x$ descriptions that is enumerated at
stage $s_{k^x_j}$ is bounded by 
the sum of $x_v$ for all $v\in [k^x_j, k^x_{j+1})$.
Let us explain this later fact in more detail.
In-between the stages $s_{k^x_j}$ and $s_{k^x_{j+1}}$ the weight of the descriptions that
are enumerated in 
$N_i^x$ is bounded by the increase in $p^x_i$, which in turn corresponds to
a limited part of $x_j$. At stage $s_{k^x_{j+1}}$
the parameter $p^x_i$ is set to $0$ and the overall weight of
$N_i^x$ descriptions that were issued since stage $s_{k^x_{j}}$ is bounded by $x_j$.

In this way the different weights of descriptions that are enumerated in
$N_i^x$ at the key stages $s_{k^x_j}$ correspond to disjoint parts of the domain
of the universal machine $U$, of larger or  equal weight. It follows that
the total weight that is enumerated in $N_i^x$ due to movements
 of $m_i$ during the construction
 is bounded by $2^{-2}$.
 If we combine this with the weight that is added by 
 applications of (\ref{eq:tsubrodouule}) we get
 $\mathtt{wgt}(N_i^x)<2^{-1}$.
\end{proof}
The following fact is crucial in showing that $\emptyset'\leq_T B$.
Its proof
uses the fact that each (version of) $N^a_i, N^d_i$ is a prefix-free machine,
which was established in Lemma \ref{le:nibodund}.
It is instructive to compare this proof with the proof of the analogous Lemma
\ref{le:miconv} of Section \ref{subse:profthA}, and identify the way that the non-uniformity
(i.e.\ the fact that we have to deal with two given sets $A, D$, and construct two corresponding
machines $M_a$ and  $M_d$) is dealt with.
\begin{lem}\label{le:micodaanva}
For each $i$, marker $m_i$ is defined, injured only finitely many times and reaches a limit.
\end{lem}
\begin{proof}
We argue by induction.
Suppose that the lemma holds for $i\in\Nat$. Then there is some stage
$s_0$ at which marker $m_i$ has stopped moving.
In order to conclude the induction step and the proof of this lemma, it suffices
to show that $m_{i+1}$ will reach a limit.
By the induction hypothesis, $m_{i+1}$
stops being injured after stage $s_0$.
Hence $c_{i+1}$ reaches a limit at $s_0$.
Since $A$  is not $K$-trivial, there is some least $n_a$ such that
$K_{N^a_{i+1}}(A\restr_{n_a})>K(n_a)+c_{i+1}$.
Similarly, since $D$  is not $K$-trivial, there is some least $n_d$ such that
$K_{N^d_{i+1}}(A\restr_{n_d})>K(n_d)+c_{i+1}$.

Let $s_1>s_0$ is a stage where 
\begin{itemize}
\item the approximations to $A\restr_{n_a}$ and $K(j)$, $j\leq n_a$ 
have settled; 
\item the approximations to $D\restr_{n_d}$ and $K(j)$, $j\leq n_d$ 
have settled.
\end{itemize}
Then the approximations to $t^a_{i+1}$, $q^a_{i+1}$
and $t^d_{i+1}$, $q^d_{i+1}$ also reach a limit
by stage $s_1$. In particular, the limit of $t^a_{i+1}$ is $n_a$
and the limit of $t^d_{i+1}$ is $n_d$. 

If marker $m_k$ moved after stage $s_1$,
this would be either due to clause (b)
or due to clause (c)
of Section \ref{subse:lackunifsol}. In the first case
the construction would enumerate an $N^a_{i+1}$-description of $A\restr_{n_a}$
of length $K(n_a)+c_{i+1}[s_0]$ 
and in the second case
an $N^d_{i+1}$-description of $D\restr_{n_d}$
of length $K(n_d)+c_{i+1}[s_0]$. 
The first action would contradict the choice of $n_a$ 
and the second action would contradict the choice of $n_d$.
Hence $m_k$ reaches a limit by stage $s_1$ and this 
concludes the induction step.
\end{proof}

As in the argument of
Section \ref{subse:profthA},
there is a  many-one correspondence between
the domain of $M_a$ and the domain of the universal machine 
$U$. We say that a $U$-description
is $A$-{\em used} if it corresponds to a string in the domain of 
$M_a$. Moreover it is $A$-{\em used} $n$ times if it corresponds to $n$
different strings in the domain of $M_a$. If a $U$-description that is already used at stage $s$
becomes used again at stage $s+1$ we say that it was {\em reused}.
Let $S^a_0$ contain the descriptions in $U$  that are $A$-used at least once.
For each $k>0$ let $S^a_k$ contain the 
descriptions in the domain of $U$ which 
are $A$-used at least $k+1$ times. Note that $S^a_{i+1}\subseteq S^a_i$ for
each $i$. According to the 
correspondence between the domains of $U$ and $M^a$, a string $\sigma$
in the domain of $U$ that is $A$-used $k$ times incurs weight $k\cdot 2^{-|\sigma|}$ 
to the domain
of $M^a$. Similar terminology and observations apply 
on $D$ and $M_d$.
Hence we have (\ref{eq:Mwdbou}).
\begin{equation}\label{eq:Mwdbou}
\textrm{$\mathtt{wgt}(M_a)\leq \sum_k  \mathtt{wgt}(S^a_k)$ \ \ \ and\ \ \ \
$\mathtt{wgt}(M_d)\leq \sum_k \mathtt{wgt}(S^d_k)$.}
\end{equation}
Note here that we avoided a multiplicative factor $k$ in the above sums.
This is not needed as each description in $S^a_k$ will be counted $k+1$ in the above sum
(and similarly with $S^d_k$). This, in turn, is a consequence of the fact that the sets in the sequences
$(S^a_k)$ and $(S^d_k)$ are nested.

A $U$-description is called $A$-{\em active}
at stage $s$ if $U(\sigma)[s]\subseteq A_s$.
By the construction,  only
currently $A$-active strings may move
from $S^a_k$ to $S^a_{k+1}$ 
 and only
currently $D$-active strings may move
from $S^d_k$ to $S^d_{k+1}$ at any given stage. 

The sets $S^a_k$ and $S^d_k$ 
may be visualized as the containers of two independent 
decanter models that are identical to the one illustrated
in Figure \ref{fig:treq}. 
If at some stage a marker $m_i$ moves (while $m_j, j<i$ remain stable)
some strings from $S^a_k$ enter $S^a_{k+1}$
and some strings from $S^d_k$ enter $S^d_{k+1}$ for various $k\in\Nat$.
In this case we say that these strings
were $A$-reused and $D$-reused respectively by $m_i$. 

The justification of the following lemma is  analogous 
to Lemma \ref{eq:bfolocoAnsk} of Section  
 \ref{subse:profthA}. However it also deals with the non-uniformity  
 that was discussed in Section \ref{subse:lackunifsol}, so it is not identical
 to the argument that was used in the proof of Lemma \ref{eq:bfolocoAnsk}.
\begin{lem}\label{eq:bfodlocoAnsk}
If during the interval of stages $[s,r]$ 
a marker $m_n$ is not injured and $n\not\in\emptyset'_r$
then the weight of the strings that are $A$-reused by $m_n$ during
this interval which remain active at stage $r$ is at most $2^{-c_n[r]}+p_n^a[r]$.
\end{lem}
\begin{proof}
Let $(s_i)$ be the sequence of stages in
$[s,r]$  where
$m_n$ moves {\em and} an enumeration into $N_n^a$ occurs.
Note that $m_n$ may move without 
an enumeration into $N_n^a$ taking place.
Moreover, at each stage $s_i$, the construction sets
$p_n^a[s_i]=0$. We claim that it suffices to show the lemma for the special 
set of stages $(s_i)$. Indeed, by the construction, the
weight of the strings that are $A$-reused by $m_n$ during a stage in $(s_i, s_{i+1})$
is bounded by the increase in $p_n^a$.
Hence if we prove that at each stage $s_i$ the weight of the strings that are $A$-reused by $m_n$
and remain active at stage $r$ is bounded by $2^{-c_n[r]}$, we also have the result of the lemma for each stage in $[s,r]$.

In order to establish at each stage $s_i$ the bound $2^{-c_n[r]}$ for the
the weight of the strings that are $A$-reused by $m_n$
and remain active at stage $r$ we will follow the argument that was given
in the proof of  Lemma \ref{eq:bfolocoAnsk}.
Note that instead of $q_n, t_n$ we now have $q^a_n, t^a_n$ and instead of
the facts (\ref{eq:timirele}), (\ref{eq:tmovmonle}) we now have
(\ref{eq:timidrele}), (\ref{eq:tmovdmonle}) respectively.

By the hypothesis of the lemma, the parameter 
$c_n$ remains constant throughout the interval $[s,r]$.
At stage $s_i\in [s,r]$ marker $m_n$ moves. 
By the definition of stages $s_{i-1}, s_{i}$, no $N_n^a$ enumeration takes
place in the interval $(s_{i-1}, s_i)$, except perhaps for
the computations from clause \eqref{eq:timirele} of the construction.
Since $m_n$ did not move during the stages in $(s_{i-1}, s_i)$ for the sake of 
clause (b) of Section \ref{subse:lackunifsol}, it follows that
\begin{equation}\label{eq:mnpaypr2evdesc}
\sum_{j>m_n[s_{i-1}]} 2^{-K(A \restr_j)[s]} < q_n[s]\ \ \ \textrm{for $s\in (s_{i-1}, s_i)$.}
\end{equation}
Note that when $m_n$ moves at stage $x+1$,  
the weight of the $U$-descriptions
that it reuses is at most
$\sum_{j>m_n[s_{i-1}]} 2^{-K(A \restr_j)[s_i-1]}$ (and not $\sum_{j>m_n[s_{i-1}]} 2^{-K(A \restr_j)[s_i]}$).
This happens because the construction first moves marker $m_n$ and then enumerates
additional computations in $M$. In other words, the descriptions that $m_n$ reuses at $s_i$ 
correspond to $M$-computations that occurred in the previous stages, not the $M$-computations
that may occur by the end of stage $s_i$. 
Hence by \eqref{eq:mnpaypr2evdesc}, 
the weight of the $U$-descriptions that are reused by $m_n$ at stage $s_i$
are bounded by $q^a_n[s_i-1]$, which is $2^{-K(t^a_n)[s_i-1]-c_n[s_i-2]}$.

Now let us consider the overall effect of the movement of $m_n$ 
during the stages $(s_i)$.
If at least one of the descriptions in $U$ that $m_n$ 
reused at some stage  $s_i$
continues to be active at stage $r$,
then  $A_{s_i}\restr_{m_n[s_i-1]+1}=A_r\restr_{m_n[s_i-1]+1}$.
By (\ref{eq:timidrele}), 
under the same assumptions this implies
\[
A_{s_i}\restr_{t^a_n[s_i-1]+1}=A_r\restr_{t^a_n[s_i-1]+1}.
\]
By ($\ast$) of the construction 
(i.e.\ the enumeration of a computation
in $N^a_i$ upon the movement of a marker), 
since at stage $s_i$ the marker $m_n$ moved, we have
$t^a_n[s_i]=t^a_n[s_i-1]+1$.
Hence by (\ref{eq:tmovdmonle}) we get that
\[
t^a_n[y]\geq t^a_n [s_i]>t^a_n[s_i-1] \ \ \ \textrm{for all $y\in [s_i, r]$.}
\]
The above observation along with the bound that we established 
in the previous paragraph
on the weight of the $U$-descriptions that are reused by $m_n$ at  a stage in
$[s, r]$ , imply the following fact.
\begin{equation*}
\parbox{10.7cm}{At the stages $(s_i)$
the weight of the descriptions in $U$
 that are $A$-used due to $m_n$ 
and remain active at stage $r$,  
are bounded by $2^{-K(t^a_n)-c_n}$,
where $t^a_n$ is larger and larger and $c_n$ remains equal to $c_n[s]$}
\end{equation*}
(while the Kolmogorov function follows its usual approximation). 
More formally, $t^a_n[s_{i}-1]< t^a_n[y_{i}]$ and the
weight of $U$-descriptions that $m_n$ reuses at stage $s_i$ and remain active 
at stage $r$ is at most
$2^{-K(t^a_n)[s_i-1]-c_n[s]}$. So the total weight of the $U$-descriptions that $m_n$ uses
during the stages in $[s, r]$ and which remain active at stage $r$ is less than
\[
\sum_i 2^{-K(i)-c_n[s]}.
\]
Since the above sum is bounded by $2^{-c_n[s]}$, this
concludes the proof.
\end{proof}
\noindent
The same argument applies symmetrically to the strings that are $D$-used,
providing the bound $2^{-c_n[s]}+p_n^d[s]$.
\begin{lem}\label{eq:bfodlocoDnsk}
If during the interval of stages $[s,r]$ a marker $m_n$ is not injured
then the weight of the strings that are $D$-reused by $m_n$ during
this interval which remain active at stage $r$ is at most $2^{-c_n[s]}+p_n^d[s]$.
\end{lem}
\noindent
Note that $p_n^a[s]\leq q_n^a[s]$ for each $n$ and all stages $s$.
This follows from clause (b)
in Section \ref{subse:lackunifsol}
and the fact that whenever $m_n$ moves
due to this clause (or is injured) parameter
$p_n^a$ takes value 0.
On the other hand by the definition of $q_n^a$ we have
$q_n^a[s]< 2^{-c_n[s]}$, so $p_n^a[s]\leq 2^{-c_n[s]}$.
Hence the bound in Lemma \ref{eq:bfodlocoAnsk} can be replaced with
$2^{-c_n[s]+1}$.
A similar argument applies to $p_n^d[s]$.
The proof of Lemma \ref{le:Mbcombound}
uses this observation in an adaptation of the proof of 
the analogous Lemma \ref{le:Mbound}.
 
\begin{lem}\label{le:Mbcombound}
The weight of the requests that are enumerated in $M_a$ is finite;
the same holds for $M_d$.
\end{lem}
\begin{proof}
We give the proof for $M_a$; the proof for $M_d$ is entirely symmetric.
According to the correspondence between the domains of $U$ and $M_a$ that we discussed,
we can bound the weight of $M_a$ via \eqref{eq:Mwdbou}. Note that each description in $S^a_k$ is
counted $k+1$ times in this sum as it belongs to all $S^a_i, i\leq k$. 
So it suffices to show that
\begin{equation}\label{eq:dboundonsk}
\texttt{wgt}(S^a_k)< 2^{-k-1} \ \ \textrm{for each $k\geq 0$.}
\end{equation}
Since only strings in the domain of $U$ are used,
$\mathtt{wgt}(S^a_0)<2^{-2}$.
Since $S^a_1\subseteq S^a_0$, condition \eqref{eq:dboundonsk} holds for $k\leq 1$.  
Let $k>1$.
Every entry of a string into $S^a_{k}$ 
is due to a marker $m_x$ which $A$-reused it when it was already in
$S^a_{k-1}$. Since $k>1$, this string entered $S^a_{k-1}$ due to another marker 
$m_y$ with $y>x\geq 0$.
Inductively, that string entered $S^a_1$ due to a marker 
$m_z$ with $z\geq k-1$. 
Fix $z$, and let $S^a_k(z)$ contain the strings in $S^a_k$ that entered $S^a_1$
due to marker $m_z$. 
Then $S^a_k=\cup_{z\geq k-1} S^a_k(z)$
and $S^a_{k+1}(z) \subseteq S^a_k(z)$ 
for each $k>1$. Hence
\[
\texttt{wgt}(S^a_k(z))\leq \sum_{z\geq k-1} \texttt{wgt}(S^a_k(z))\ \ \textrm{for each $k>1$.}
\]
So in order to prove \eqref{eq:dboundonsk} for $k>1$ it suffices to show that
\begin{equation}\label{Szkboundfd}
\texttt{wgt}(S^a_k(z))< 2^{-z-2}\ \ \textrm{for each $z\geq 0$.} 
\end{equation}
Let $(s_i)$ be the increasing sequence of 
stages where $m_z$
is injured. Note that at this point we do not assume that
$(s_j)$ is a finite sequence.
We may count the weight of $S_k^a(z)$
by counting the weight of the bunches of descriptions that
enter in $S^a_1(z)$ and then  enter in $S^a_2(z)$
(necessarily by some $m_j$ with $j<z$). 
This is justified because every description that enters
$S_k^a(z)$ must have passed from $S_2^a(z)$ first.

Since the movement of a marker $m_i$ injures all $m_j, j>i$,
the only stages were strings move from 
$S^a_1(z)$ to $S^a_{2}(z)$
are the stages $(s_i)$.
Moreover 
since only active strings move from
$S^a_1(z)$ to $S^a_{2}(z)$ at stage $s_i$, according to Lemma \ref{eq:bfodlocoAnsk}
(and the observation straight after it)
their weight is bounded by $2^{-c_z[s_i-1]+1}$.
So the weight of the strings that enter
$S^a_{2}(z)$ from $S^a_1(z)$ is bounded above by $\sum_j 2^{-c_z[s_j-1]}$.
Since $c_z[s_{j+1}-1]=c_z[s_j]<c_z[s_{j}-1]$ for
all $j$,  this weight is bounded by $\sum_j 2^{-c_z[0]-j}=2^{-c_z[0]+1}$.
Since $c_z[0]=z+4$ this bound becomes
$2^{-z-2}$, which establishes \eqref{Szkboundfd} and concludes the proof.
\end{proof}
\noindent
We conclude with the proof that (\ref{eq:neduagreqkb}) is met.
\begin{lem}\label{le:micodaanv}
The following hold for each $i$.:
\begin{itemize}
\item there is an $M_a$-description of $B\restr_i$ of length $\leq K(A\restr_i)$;
\item there is an $M_d$-description of $B\restr_i$ of length $\leq K(D\restr_i)$.
\end{itemize}
\end{lem}
\begin{proof}
We argue by induction on $i$.
Suppose that the lemma holds for $i\in\Nat$. Then there is some stage
$s_0$ at which marker $m_i$ is defined and has stopped moving and for each
$(X,x)\in\{(A,a), (D,d)\}$
\begin{itemize}
\item the approximations to $X\restr_{i+1}, B\restr_{i+1}$, 
$K(X\restr_{i+1})$,  $K_{M_x}(B\restr_i)$  have settled;
\item $K_{M_x}(B\restr_i)[s_0]\leq K(X\restr_i)[s_0]$.
\end{itemize}
For each
$(X,x)\in\{(A,a), (D,d)\}$,
If $K_{M_x}(B\restr_{i+1})[s_0]> K(X\restr_{i+1})[s_0]$ the construction
at stage $s_0+1$ will enumerate
an $M_x$-computation that describes
$B\restr_{i+1}$ with a string of length $K(X\restr_{i+1})$.
\end{proof} 
\noindent
By Lemma \ref{le:micodaanva} and the construction we get
that the movement of the markers satisfies properties 
(i)-(v) of Section \ref{se:prothlowisa}.
Hence $\emptyset'\leq_T B$. We conclude the proof of Theorem
\ref{th:lowiscom}
by observing that (\ref{eq:neduagreqkb}) is met.
By Lemma \ref{le:micodaanv} the construction enumerates the required requests in $M_a$
which ask for a description of $B\restr_i$ 
with a string of length at most $K(A\restr_i)$, for each $i$.
Moreover the same holds for $D$ in place of $A$ and $M_d$ in place of $M_a$.
On the other hand Lemma \ref{le:Mbcombound} establishes that these request sets correspond to
prefix-free machine, via the Kraft-Chaitin lemma. Hence \eqref{eq:neduagreqkb} is met,
which concludes the verification of the construction
and the proof of Theorem \ref{subse:profthB}.

\section{Concluding remarks}\label{se:concrem}
We have demonstrated that computably enumerable sets can have a lot of information (for example, a solution to the
halting problem) yet have very simple initial segments. On the other hand, as we discussed, it is known that
such sets cannot have trivial initial segment complexity. In other words, their initial segments are more complex that
the initial segments of an infinite sequence of 0s. Our result has had numerous applications, which were discussed in
Section \ref{subse:applic}.

The methods that we used have novel features, but are not completely new. The bulk of the argument is depicted
in Figure \ref{fig:dynam2} which indicates the dynamic relationships between each pair of the three pairs from the following actions:
\begin{itemize}
\item[(a)] bound the complexity constructed set;
\item[(b)] challenge the non-triviality of the given set;
\item[(c)] code information into the constructed set.
\end{itemize}
After some abstraction, this type of argument can be found in other places in the recent literature (some times in simpler forms)
where a set with non-trivial algorithmic-theoretic complexity is given and one is required to construct a set with lesser complexity
which encodes certain kinds of information. 
Examples of such arguments can be found in \cite{BarmpaliasM09, Barmpalias:08, BarmpaliasCompress, omunca}.
However in the present paper we have made a conscious effort to explain the intuition and the dynamics of the argument
in concrete terms. Despite the common form of these arguments, however, each case has its own unique features that
stem from the particular measures of complexity that are involved. As an example in the $LK$-degrees, 
in \cite{Barmpalias:08} it was shown that
every non-zero $\Delta^0_2$ degree has uncountably many predecessors and 
in \cite{BarmpaliasCompress} it was shown that there are no
minimal pairs of $\Delta^0_2$ degrees. However, as we discussed, in the $K$-degrees every c.e.\ degree
has only countably many predecessors. Moreover, although we showed that there is no minimal pair of $K$-degrees of c.e.\ sets, 
the same question for $\Delta^0_2$ sets remains open.

%
%
\end{document}